\documentclass[11pt,a4paper]{article}
\newcommand{\TemplateStyle}{manuscript}
\providecommand{\TemplateStyle}{manuscript}
\usepackage{etoolbox}
\newif\ifTemplateManuscript
\newif\ifTemplateProject
\newif\ifTemplateChinese
\ifdefstring{\TemplateStyle}{manuscript}{\TemplateManuscripttrue}{%
  \ifdefstring{\TemplateStyle}{project}{\TemplateProjecttrue}{%
    \ifdefstring{\TemplateStyle}{chinese}{\TemplateChinesetrue}{%
      \PackageError{layout_settings}{Unknown TemplateStyle}{Use manuscript, project, or chinese.}}}}
\providecommand{\ProjectHeader}{Research Plan}
\providecommand{\TemplateLegacyMath}{true}

\usepackage{iftex}
\ifPDFTeX
  \usepackage[utf8]{inputenc}
  \usepackage{CJKutf8}
\else
  \ifXeTeX
    \usepackage{fontspec}
    \usepackage{xeCJK}
    \IfFontExistsTF{DFKai-SB}
      {\setCJKmainfont{DFKai-SB}[AutoFakeBold=2]
       \setCJKsansfont{DFKai-SB}[AutoFakeBold=2]}
      {\IfFontExistsTF{BiauKai}
        {\setCJKmainfont{BiauKai}[AutoFakeBold=2]
         \setCJKsansfont{BiauKai}[AutoFakeBold=2]}
        {
         \IfFileExists{fonts/bkai00mp.ttf}{%
           \setCJKmainfont{bkai00mp.ttf}[Path=fonts/,AutoFakeBold=2,AutoFakeSlant=.2]
           \setCJKsansfont{bkai00mp.ttf}[Path=fonts/,AutoFakeBold=2,AutoFakeSlant=.2]
         }{%
           \IfFontExistsTF{AR PL KaitiM Big5}{%
             \setCJKmainfont{AR PL KaitiM Big5}[AutoFakeBold=2,AutoFakeSlant=.2]
             \setCJKsansfont{AR PL KaitiM Big5}[AutoFakeBold=2,AutoFakeSlant=.2]
           }{\PackageError{layout_settings}{Traditional Chinese font missing}{%
             Upload the supplied fonts folder, or compile with pdfLaTeX.}}}}}
  \else
    \PackageError{layout_settings}{Unsupported engine}{Use pdfLaTeX or XeLaTeX.}
  \fi
\fi
\ifPDFTeX
  \AtBeginDocument{\begin{CJK*}{UTF8}{bkai}}
  \AtEndDocument{\clearpage\end{CJK*}}
\fi

\usepackage{amsmath}
\usepackage{amsfonts}
\usepackage{amssymb}
\usepackage{amsthm}
\usepackage{bbm}
\usepackage{esint}
\usepackage{mathrsfs}
\usepackage[scr=stixtwofancy,
    scrscaled=1.15,
    bb=bboldx,
    bbscaled=1.05]{mathalpha}

\usepackage[
    a4paper,
    left=3.18cm,
    right=3.18cm,
    top=2.54cm,
    bottom=2.54cm
]{geometry}
\usepackage{setspace}
\usepackage{enumitem}
\usepackage{comment}
\usepackage{lineno}

\usepackage{xcolor}
\usepackage{graphicx}
\ifPDFTeX\usepackage{epstopdf}\fi
\usepackage{overpic}
\usepackage{subcaption}
\usepackage{tikz}

\usepackage{booktabs}
\usepackage{multirow}
\usepackage{rotating}
\usepackage{algorithm}
\usepackage{algpseudocode}

\ifTemplateManuscript
\usepackage{titling}

\usepackage{authblk}
\usepackage[nottoc,numbib]{tocbibind}
\fi
\usepackage{titlesec}
\usepackage{array,tabularx,float,fancyhdr,lastpage}
\ifTemplateChinese\usepackage{microtype}\fi

\usepackage{hyperref}
\usepackage[nameinlink]{cleveref}

\definecolor{trueblue}{rgb}{0.0, 0.45, 0.81}
\definecolor{lightgoldenrodyellow}{rgb}{0.98, 0.98, 0.82}
\definecolor{uclablue}{rgb}{0.33, 0.41, 0.58}
\definecolor{darkcerulean}{rgb}{0.03, 0.27, 0.49}
\definecolor{carnelian}{rgb}{0.7, 0.11, 0.11}
\definecolor{lightgray}{rgb}{.25,.25,.25}
\definecolor{midgray}{rgb}{.6,.6,.6}
\definecolor{bleudefrance}{rgb}{0.19, 0.55, 0.91}
\definecolor{columbiablue}{rgb}{0.61, 0.87, 1.0}
\definecolor{oxfordblue}{rgb}{0.0, 0.13, 0.28}
\definecolor{yaleblue}{rgb}{0.06, 0.3, 0.57}
\definecolor{harvardcrimson}{rgb}{0.79, 0.0, 0.09}
\definecolor{burgundy}{rgb}{0.50, 0.0, 0.13}
\definecolor{forestgreen}{rgb}{0.0, 0.35, 0.18}
\definecolor{darkelectricblue}{rgb}{0.33, 0.41, 0.47}
\definecolor{babyblueeyes}{rgb}{0.63, 0.79, 0.95}
\definecolor{ceruleanblue}{rgb}{0.16, 0.32, 0.75}
\definecolor{sphereblue}{RGB}{129,159,247}
\definecolor{bananayellow}{rgb}{1.0, 0.88, 0.21}
\definecolor{MSBlue}{rgb}{.204,.353,.541}

\numberwithin{equation}{section}
\numberwithin{figure}{section}
\theoremstyle{plain}
\newtheorem{theorem}{\sffamily Theorem}[section]

\newtheorem{lemma}{\sffamily Lemma}[section]

\setlist[enumerate]{leftmargin=.5in}
\setlist[itemize]{leftmargin=.5in}

\graphicspath{{images/}{./}}
\usetikzlibrary{
    arrows.meta,
    calc,
    positioning,
    angles,
    quotes,
    intersections,
    3d,
    perspective,
    decorations.pathreplacing,
    decorations.markings
}

\newlength{\Oldarrayrulewidth}

\hypersetup{
    colorlinks=true,
    allcolors=trueblue,
    hypertexnames=false
}
\crefname{section}{section}{sections}
\crefname{subsection}{subsection}{subsections}
\Crefname{section}{Section}{Sections}
\Crefname{subsection}{Subsection}{Subsections}
\Crefname{figure}{Figure}{Figures}
\Crefname{corollary}{Corollary}{Corollary}

\titleformat*{\section}{\Large\bfseries\sffamily\color{black}}
\titleformat*{\subsection}{\large\bfseries\sffamily\color{black}}
\titleformat*{\subsubsection}{\bfseries\sffamily\color{black}}

\makeatletter
\def\@fnsymbol#1{\ensuremath{\ifcase#1\or \dagger\or \ddagger\or
   \mathsection\or \mathparagraph\or \|\or **\or \dagger\dagger
   \or \ddagger\ddagger \else\@ctrerr\fi}}
\makeatother

\newcommand{\argmin}{\operatornamewithlimits{argmin}}

\ifdefstring{\TemplateLegacyMath}{true}{%
\def\0{\mathbf{0}}
\def\1{\mathbf{1}}

\def\f{\mathbf{f}}
\def\g{\mathbf{g}}
\def\h{\mathbf{h}}

\def\p{\mathbf{p}}

\def\u{\mathbf{u}}
\def\v{\mathbf{v}}
\def\w{\mathbf{w}}

\def\y{\mathbf{y}}

\def\R{\mathbb{R}}

\def\bf{\mathbbm{f}}

\def\K{\mathcal{K}}
\def\M{\mathcal{M}}

\def\rI{\mathrm{I}}

\def\trace{\mathrm{trace}}

\def\B{\mathtt{B}}

\def\I{\mathtt{I}}

\def\bphi{{\boldsymbol{\phi}}}

\def\btheta{{\boldsymbol{\theta}}}

\ifTemplateProject\fi
}{}
\definecolor{heading}{RGB}{42,68,86}
\definecolor{placeholdergray}{RGB}{90,90,90}

\ifTemplateProject
  \setlist[itemize]{leftmargin=2.2em,itemsep=0.25em,topsep=0.35em}
  \setlist[enumerate]{leftmargin=2.4em,itemsep=0.25em,topsep=0.35em}
  \titleformat{\section}{\Large\bfseries\sffamily}{\thesection.}{0.7em}{}
  \titlespacing*{\section}{0pt}{1.2em}{0.65em}
  \titleformat{\subsection}{\large\bfseries\sffamily}{\thesubsection}{0.65em}{}
  \titlespacing*{\subsection}{0pt}{1em}{0.45em}
  \titleformat{\subsubsection}{\normalsize\bfseries\sffamily}{\thesubsubsection}{0.6em}{}
  \titlespacing*{\subsubsection}{0pt}{0.8em}{0.3em}
  \hypersetup{allcolors=trueblue}
\fi
\ifTemplateChinese
  \setlist[itemize]{leftmargin=2.2em,itemsep=0.2em,topsep=0.25em}
  \titleformat{\section}{\large\bfseries\color{heading}}{\thesection.}{0.55em}{}
  \titlespacing*{\section}{0pt}{1.05em}{0.35em}
  \hypersetup{allcolors=trueblue}

\fi

\def\bv{\mathbbm{v}}

\def\rV{\mathrm{V}}

\title{\LARGE\bfseries\sffamily Volume-Preserving Parameterizations via Preconditioned Nonlinear Conjugate Gradient Method}

\author[1]{Shu-Yung Liu}
\author[1]{Tsung-Ming Huang}
\author[2,3,4]{Wen-Wei Lin}
\author[1]{Mei-Heng Yueh}
\affil[1]{Department of Mathematics, National Taiwan Normal University, Taipei, Taiwan (\href{mailto:lii227857@gmail.com}{lii227857@gmail.com}, \href{mailto:min@ntnu.edu.tw}{min@ntnu.edu.tw}, \href{mailto:yue@ntnu.edu.tw}{yue@ntnu.edu.tw})}
\affil[2]{Nanjing Center for Applied Mathematics, Nanjing, China}
\affil[3]{Shanghai Institute for Mathematics and Interdisciplinary Sciences, Shanghai, China}
\affil[4]{Department of Applied Mathematics, National Yang Ming Chiao Tung University, Hsinchu, Taiwan (\href{mailto:wwlin@math.nctu.edu.tw}{wwlin@math.nctu.edu.tw})}
\date{}
\hypersetup{pdftitle={Volume-Preserving Parameterizations via Preconditioned Nonlinear Conjugate Gradient Method},
  pdfauthor={Shu-Yung Liu, Tsung-Ming Huang, Wen-Wei Lin, Mei-Heng Yueh}}

\begin{document}
\maketitle
\begin{abstract}\sffamily
A volume-preserving parameterization is a bijective mapping that maps a 3-manifold onto a canonical domain while preserving local volume. We formulate this problem as an unconstrained nonlinear optimization problem by introducing an isovolumetric energy that quantifies volumetric distortion. This energy is minimized using a preconditioned nonlinear conjugate gradient method, which is globally convergent when the line search satisfies the strong Wolfe conditions. Numerical experiments demonstrate that the proposed method achieves improved accuracy and efficiency, and it supports applications in shape registration and volumetric deformation.
\par\bigskip
\textbf{Keywords.} Tetrahedral mesh, Simplicial mapping, Volume-preserving parameterization, Nonlinear conjugate gradient method
\par\medskip
\textbf{Mathematics Subject Classification.} 65D18, 68U05, 68U01, 65D17
\end{abstract}

\section{Introduction}

A parameterization of a simply connected 3-manifold is a bijective map from the manifold to a simple canonical domain (e.g., a unit ball or cube), thereby simplifying computation on complex geometric structures for subsequent processing. Numerous volumetric mapping methods have been developed, including close-to-conformal deformations \cite{ChPS15}, the simplicial foliation method \cite{CaSZ16}, scalable locally injective mappings \cite{RaPP17}, the simplex lifting method \cite{DuAZ20}, foldover-free mappings \cite{GaKK21}, and symmetric volume mappings \cite{AbSG23}.

A volume-preserving parameterization is a parameterization that preserves the local volume of a given 3-manifold. It has been applied to medical imaging tasks such as brain tumor segmentation~\cite{LiJY21,LiLH22}. In this application, a volume-preserving map serves as a preprocessing step that maps the brain to a unit cube, reducing the raw image size by removing redundant space outside the brain. As a result, the subsequent neural network training operates on a cube representation of the brain, improving both training efficiency and accuracy in identifying tumor regions.

Several methods have been proposed to compute such mappings. Gu et al. \cite{GuLS16,SuCL17}, based on the theory of convex polyhedra, compute a polyhedral decomposition of the unit ball such that the volume of each polyhedral cell matches the vertex volume of the given tetrahedral mesh. The resulting map, which sends each cell to its corresponding vertex, is a discrete optimal transport map. Therefore, it can be represented as the gradient of a piecewise linear convex function, and the associated convex optimization problem can be solved by Newton's method.

Alternatively, Choi et al. \cite{ChCh21,LyCC25} treat densities as local volume ratios and adopt the density-equalizing map framework. The density is evolved according to the heat equation, which induces a velocity field that can be integrated to recover the vertex positions. The heat equation is solved via the implicit Euler method, and its equilibrium state yields the desired volume-preserving parameterization on the unit ball.

On the other hand, Yueh et al.~\cite{YuLL19,YuLL20,HuLL23,TaLL25} introduced the volumetric stretch energy, a nonlinear functional that measures local volume distortion, whose minimization yields the desired volume-preserving mapping subject to a fixed total-volume constraint. Their approach typically fixes the boundary and updates the interior vertices via a fixed-point iteration. While it converges rapidly in practice, each step requires solving a linear system, which can be computationally expensive on large tetrahedral meshes. Moreover, from a theoretical standpoint, global convergence of the fixed-point iteration is not generally guaranteed. As a result, nonlinear conjugate-gradient (CG) methods, which have low memory requirements and well-established convergence properties under suitable line-search conditions, can be preferable for large-scale problems.

The CG method was originally proposed for symmetric positive definite linear systems \cite{HeSt52}. It improves the gradient descent method by incorporating an additional correction term, yielding faster convergence and finite termination. It was later extended to nonlinear unconstrained optimization~\cite{FlRe64}, and global convergence has since been established for several variants under step sizes satisfying the Wolfe conditions~\cite{ALBA85,LiHY95,DaYu96}.  However, the nonlinear CG method can suffer from inefficiency in practice. In particular, it is susceptible to jamming, where many short steps are taken with limited progress. To address this, various modifications to the correction parameter have been proposed~\cite{PoRi69,Flet00,DaYu99,HaZh05}. For additional background and details, the reader is referred to~\cite{DaHL00,Hage06}.

In this paper, we compute ball-shaped volume-preserving parameterizations by minimizing local volume distortion with the nonlinear CG method. Unlike the fixed-point iterations that require a fixed boundary condition, the nonlinear CG method permits boundary gliding on the unit sphere by incorporating spherical coordinates. Since boundary displacement can alter the total volume, we generalize the volumetric stretch energy by introducing an isovolumetric energy as the objective function, thereby relaxing the total volume constraint. Furthermore, this formulation provides a suitable symmetric positive definite preconditioner that significantly improves convergence while maintaining global convergence under strong Wolfe line searches, achieving state-of-the-art accuracy in volume preservation.

\subsection{Contribution}
The primary contributions of this paper are listed as follows:
\begin{itemize}
\item[(i)] We formulate the ball-shaped volume-preserving parameterization as an isovolumetric energy minimization (IEM), which allows boundary vertices to be optimized along the unit sphere by utilizing spherical coordinates.

\item[(ii)] We solve the IEM problem by preconditioning the nonlinear CG method, which is globally convergent. The proof can be regarded as a generalization of the convergence results for the nonlinear CG method presented in~\cite[Chapter 5]{NoWr06}.

\item[(iii)] Numerical experiments indicate that the effectiveness of reducing volume distortion is significantly improved compared to the previous state-of-the-art method. 

\item[(iv)] We also provide applications to shape registration and deformation, which demonstrate the practical utility of the proposed method. 
\end{itemize}

\subsection{Notation}
In this paper, we adopt the following notation.
\begin{itemize}
\item $\R$ represents the set of real numbers.
\item $e$ denotes the identity mapping.
\item Bold letters, such as $\mathbf{f}$ and $\mathbf{g}$, indicate real-valued vectors.
\item Capital letters and blackboard bold lowercase letters, such as $L$ and $\bf$, represent real-valued matrices.
\item Typewriter-style letters, such as $\mathtt{I}$ and $\mathtt{B}$, are used to denote ordered sets of indices.
\item $\mathbf{f}_i$ refers to the $i$-th entry of the vector $\mathbf{f}$.
\item $\mathbf{f}_\mathtt{I}$ denotes the subvector of $\mathbf{f}$ consisting of entries $\mathbf{f}_i$ for $i \in \mathtt{I}$.
\item $L_{i,j}$ represents the $(i,j)$-th entry of the matrix $L$.
\item $L_{\mathtt{I},\mathtt{J}}$ denotes the submatrix of $L$, composed of entries $L_{i,j}$ for $i \in \mathtt{I}$ and $j \in \mathtt{J}$.
\end{itemize}

\subsection{Organization of the paper}

The remainder of this paper is structured as follows. In Section~\ref{sec:2}, we introduce simplicial 3-manifolds, mappings, and the isovolumetric energy functional. Section~\ref{sec:3} formulates the IEM problem for ball-shaped mappings and provides the explicit formula for the energy gradient. In Section~\ref{sec:4}, we propose the preconditioned nonlinear CG method for solving the IEM problem. The convergence analysis is presented in Section~\ref{sec:5}. Section~\ref{sec:6} includes numerical experiments and comparisons with state-of-the-art methods. Applications to volume-preserving registration and deformation are demonstrated in Section~\ref{sec:7}, followed by concluding remarks in Section~\ref{sec:8}.

\section{Simplicial 3-manifolds, mappings, and functionals}
\label{sec:2}

Simplicial 3-manifolds, mappings, and the isovolumetric energy functional are introduced in Subsections \ref{sec:2.1}, \ref{sec:2.2}, and \ref{sec:2.3}, respectively.

\subsection{Simplicial 3-manifolds}
\label{sec:2.1}

A simplicial 3-manifold $\mathcal{M}$ is the underlying set of a homogeneous simplicial 3-complex $\K(\M) = \bigcup_{\ell=0}^3 \K^\ell(\M)$ consisting of $n$ vertices
\begin{equation*}
\K^0(\M) = \big\lbrace v_s = (v_s^1, v_s^2, v_s^3)^{\top} \in \R^3 \big\rbrace_{s=1}^n,
\end{equation*}
$m$ tetrahedra
\begin{equation*}
\K^3(\M) = \lbrace \tau_t = [v_{i_t}, v_{j_t}, v_{k_t}, v_{\ell_t}] \subset \R^3 \,|\, v_{i_t}, v_{j_t}, v_{k_t}, v_{\ell_t} \in \K^0(\M) \rbrace_{t=1}^m,
\end{equation*}
triangular faces
\begin{equation*}
\K^2(\M) = \lbrace [v_i, v_j, v_k] \,| \, [v_i, v_j, v_k, v_\ell] \in \K^3(\M) \text{ for some } v_\ell \in \K^0(\M) \rbrace,
\end{equation*}
and edges
\begin{equation*}
\K^1(\M) = \lbrace [v_i, v_j] \,|\, [v_i, v_j, v_k] \in \K^2(\M) \text{ for some } v_k \in \K^0(\M) \rbrace,
\end{equation*}
where $\left[{v}_0, \ldots, {v}_k\right]$ denotes the $k$-simplex with vertices ${v}_0, \ldots, {v}_k$. Simplicial 3-manifolds are known as tetrahedral meshes, commonly used to represent solid shapes in computer graphics and computational mechanics.

\subsection{Simplicial mappings}
\label{sec:2.2}
A simplicial mapping $f:\M \to \R^3$ is a specific piecewise affine mapping such that its image $f(\mathcal{M})$ forms the underlying set of a simplicial 3-complex. 
For each tetrahedron $\tau = [v_i, v_j, v_k, v_\ell] \in \K^3(\M)$, the restriction map $f|_\tau$ is determined by the images of its vertices,
\[
f_i \equiv f(v_i) = (f^1_i, f^2_i, f^3_i)^\top, \quad v_i \in \K^0(\M). 
\]
For any $v \in \tau$, the restriction $f|_{\tau}(v)$ can be written in barycentric form as
$$
f|_{\tau}(v) = \frac{|[v,v_{j},v_{k},v_{\ell}]|}{|\tau|} f_{i} + \frac{|[v_{i},v,v_{k},v_{\ell}]|}{|\tau|} f_{j} + \frac{|[v_{i},v_{j},v,v_{\ell}]|}{|\tau|} f_{k} + \frac{|[v_{i},v_{j},v_{k},v]|}{|\tau|} f_{\ell},
$$
where the signed volume $|\tau| = |[v_i, v_j, v_k, v_\ell]|$ is defined as
\begin{equation*}
|[v_i, v_j, v_k, v_\ell]| = \frac{1}{6}  ((v_j-v_i)\times(v_k-v_i))^\top(v_\ell-v_i).
\end{equation*}

A simplicial mapping \(f\) is said to be orientation-preserving if
$$
|f(\tau)| > 0, \quad \text{for every $\tau \in \mathcal{K}^3(\mathcal{M})$},
$$
and orientation-reversing if
$$
|f(\tau)| < 0, \quad \text{for every $\tau \in \mathcal{K}^3(\mathcal{M})$}.
$$
A simplicial mapping $f$ is said to be volume-preserving if the volume of each transformed tetrahedron is proportional to the original one with a global constant scaling factor $c \in \mathbb{R}$, i.e.,
$$
|f(\tau)| = c\,|\tau|, \quad \text{for every $\tau \in \mathcal{K}^3(\mathcal{M})$}.
$$
In this case, the mapping $f$ is orientation-preserving if $c>0$, and orientation-reversing if $c<0$.

\subsection{Isovolumetric energy functional}
\label{sec:2.3}

In this paper, the adopted objective functional for computing volume-preserving parameterizations is the isovolumetric energy, defined as
\begin{subequations} \label{eq:Ea}
\begin{equation} \label{eq:Ea1}
E_\rI(f) = \frac{\mathcal{V}({e})}{\mathcal{V}(f)} E_\rV(f) -  \mathcal{V}(f), 
\end{equation}
where $\mathcal{V}$ denotes the image volume measurement, ${e}$ denotes the identity mapping on $\M$, and $E_\rV$ denotes the volumetric stretch energy functional \cite{HuLL23} defined as
\begin{equation} \label{eq:Es0}
E_\rV(f) = \sum_{\tau\in\K^3(\M)} \frac{|f(\tau)|^2}{|\tau|}.
\end{equation}
\end{subequations}
The definition of $E_\rV$ in \eqref{eq:Es0} is equivalent to \cite[(3.3)]{HuLL23} with a scaling factor of $2/3$. 
Furthermore, for orientation-preserving mappings, $E_\rI$ admits a continuous extension to degenerate cases. In this setting, the image volume is the sum of positive tetrahedral volumes, and hence $E_V(f) \to 0$ as $\mathcal{V}(f) \to 0$. We therefore define $E_\rI(f) = 0$ whenever $\mathcal{V}(f) = 0$.

The volumetric stretch energy $E_\rV$ yields volume-preserving maps when minimized under a fixed total-volume constraint. Specifically, if $\mathcal{V}(e) = \mathcal{V}(f)$, then $E_\rV$ is minimized by volume-preserving mappings $f$, as shown in \cite[Theorem~3.3]{HuLL23}. Without this constraint, however, $E_\rV$ admits only the minimum of zero, achieved by a constant map $f \equiv c$.

For a ball-shaped mapping, the resulting image volume \( \mathcal{V}(f) \) corresponds to a polyhedron that approximates the unit ball, whose volume depends on the boundary vertices. As the boundary vertices glide along the unit sphere during the iterative minimization of \( E_\rV \), the image volume \( \mathcal{V}(f) \) changes, making it undesirable to enforce the condition \( \mathcal{V}(e) = \mathcal{V}(f) \). To address this issue, we replace $E_\rV$ with the isovolumetric energy $E_\rI$, which removes the need for the total image volume constraint while still offering a similar characterization without the volume restriction, as stated in the theorem below.

\begin{theorem} \label{cor:2.2}
Let $\M$ be a simplicial $3$-manifold and $f$ be an orientation-preserving simplicial mapping defined on $\M$. The isovolumetric energy in \eqref{eq:Ea} satisfies
$$
E_\mathrm{I}(f) \geq 0,
$$
with equality occurring only when $f$ is volume-preserving.
\end{theorem}
\begin{proof}
Applying the Cauchy–Schwarz inequality to \eqref{eq:Es0}, we get
\begin{align*}
E_\mathrm{V}(f) \, \mathcal{V}({e}) &= \sum_{\tau\in\K^3(\M)} \frac{|f(\tau)|^2}{|\tau|} 
\sum_{\tau\in\K^3(\M)} |\tau|
\geq \left(\sum_{\tau\in\K^3(\M)} |f(\tau)| \right)^2 = \mathcal{V}(f)^2. 
\end{align*}
Dividing both sides of the inequality by $\mathcal{V}(f)$, we obtain
$$
\frac{\mathcal{V}({e})}{\mathcal{V}(f)}E_\mathrm{V}(f) \geq \mathcal{V}(f),
$$
which means the isovolumetric energy \eqref{eq:Ea} satisfies $E_\mathrm{I}(f) \geq 0$. 
Equality holds only when $\frac{|f(\tau)|}{|\tau|}$ is constant, i.e., when $f$ is volume-preserving.
\end{proof}

Theorem~\ref{cor:2.2} shows that $E_\rI$ generalizes $E_\rV$ by relaxing the volume constraint. This, in turn, allows us to compute volume-preserving mappings while letting the boundary vertices move along the unit sphere.

Recall that a simplicial map $f$ is fully determined by its vertex images.
In practice, we represent $f$ in matrix form as
\begin{equation} \label{eq:f}
\bf \equiv 
\begin{bmatrix}
f_1^{\top} \\
\vdots \\
f_n^{\top}
\end{bmatrix} 
=
\begin{bmatrix}
f_1^1  &  f_1^2  &  f_1^3\\
\vdots &  \vdots &  \vdots \\
f_n^1  &  f_n^2  &  f_n^3
\end{bmatrix} \equiv
\begin{bmatrix}
\f^1  &  \f^2  &  \f^3 
\end{bmatrix} \in\R^{n\times 3}.
\end{equation}
Under this representation, the volumetric stretch energy $E_\rV$ can be expressed as
\begin{subequations} \label{eq:Es}
\begin{equation} 
E_\rV(f) = \frac{1}{3} \trace (\bf^{\top} L_\rV(f) \, \bf) = \frac{1}{3} \sum_{s=1}^3 {\f^s}^{\top} L_\rV(f) \, \f^s,
\end{equation}
where $\bf$ and $\f^s$ are defined in \eqref{eq:f}, and $L_\rV(f)$ is the weighted Laplacian matrix defined as
\begin{equation} \label{eq:VSLaplacian}
[L_\rV(f)]_{i,j}=
\begin{cases}
\displaystyle -\sum_{[v_{i},v_{j},v_{k}, v_{\ell}]\in \K^3(\M)} [\omega_\mathrm{V}(f)]_{i,j}^{k,\ell}  & \mbox{if $[v_{i}, v_{j}]\in\K^1(\M)$}, \\
\displaystyle -\sum_{[v_{i}, v_{k}]\in\K^1(\M)} [L_\mathrm{V}(f)]_{i,k}   &\mbox{if $j=i$},\\
0 & \mbox{otherwise.}
\end{cases}
\end{equation}
The mapping-dependent weight $\omega_\mathrm{V}(f)$ is the modified cotangent weight defined as
\begin{equation}
[\omega_\rV(f)]_{i,j}^{k,\ell}=\dfrac{\cot{(\theta_{i,j}^{k,\ell}(f))}\,|[f_k, f_\ell]|\,|[f_{i},f_{j},f_{k},f_{\ell}]|}{6\,|[v_{i},v_{j},v_{k},v_{\ell}]|},
\end{equation}
where $\theta_{i,j}^{k,\ell}(f)$ denotes the dihedral angle between $[f_{i},f_{k},f_{\ell}]$ and $[f_{j},f_{\ell},f_{k}]$ in $[f_{i}, f_{j}, f_{k}, f_{\ell}]$ on the edge $[f_{k},f_{\ell}]$, as illustrated in Figure \ref{dihedral angle}.
The detailed derivation can be found in \cite[Theorem~3.1]{HuLL23}.

With this representation of $E_\rV$, the computation of a ball-shaped volume-preserving mapping can be formulated as an isovolumetric energy minimization (IEM) problem with respect to $\f^s$, which is an unconstrained nonlinear optimization problem, as described in Section~\ref{sec:3}.
\end{subequations}

\begin{figure}[tbp]
\centering
\includegraphics[height = 3.5cm]{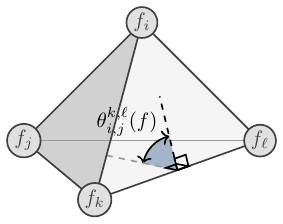}
\caption{An illustration of the dihedral angle between $ [f_{i},f_{k},f_{\ell}]$ and $[f_{j},f_{\ell},f_{k}]$ in the tetrahedron $[f_i, f_j, f_k, f_\ell]$.}
\label{dihedral angle}
\end{figure}

\section{IEM for ball-shaped parameterizations}
\label{sec:3}

Recall that a ball-shaped parameterization of a simplicial 3-manifold is a bijective map whose boundary vertices lie on the unit sphere. We enforce this boundary constraint by expressing the boundary vertices in spherical coordinates.

The index sets of boundary and interior vertices are denoted as
$$
\B = \{ b \mid v_b \in \partial\M \}
~\text{ and }~
\I = \{ i \mid v_i \in \M\backslash\partial\M \},
$$
respectively, and we denote $n_\I = \# \I$ and $n_\B = \#\B$. 
Then, the boundary mapping $\f_\B^1, ~\f_\B^2, ~\f_\B^3$ can be represented in spherical coordinates as
\begin{equation}
\f^1_\B = \sin \btheta \odot \cos \bphi, ~~ \f^2_\B = \sin \btheta \odot \sin \bphi, ~~\text{and}~~ \f^3_\B = \cos \btheta,
\label{eq:sphere_coor}
\end{equation}
where $\btheta = (\theta_1, \ldots, \theta_{n_\B})^{\top} \in\R^{n_\B}$ and $\bphi = (\phi_1, \ldots, \phi_{n_\B})^{\top} \in\R^{n_\B}$.
Conversely, the spherical coordinates $(\btheta, \bphi)$ can be recovered from $\f_\B^1, ~\f_\B^2, ~\f_\B^3$ as
\begin{equation}
\btheta = \arccos(\f_\B^3), \quad \bphi = \mathrm{atan2}(\f_\B^2, \f_\B^1),
\label{eq:sphere_coor_inv}
\end{equation}
where $\mathrm{atan2}$ denotes the four-quadrant inverse tangent. 
With this boundary expression, the IEM problem for ball-shaped parameterizations becomes the unconstrained nonlinear optimization:
\begin{equation} \label{eq:IEM}
f^* = \argmin_{\begin{array}{c}\f_\I^1, \f_\I^2, \f_\I^3\in\R^{n_\I}\\ \btheta, \bphi\in\R^{n_\B}\end{array}} E_\rI(\f_\I^1, \f_\I^2, \f_\I^3, \btheta, \bphi).
\end{equation}
We next derive the gradient of $E_\rI$ in \eqref{eq:IEM} with respect to $\f_\I^s$, $\btheta$, and $\bphi$ for $s=1,2,3$.

Since the image volume $\mathcal{V}(f)$ depends only on the boundary vertices, it suffices to express $\mathcal{V}(f)$ in terms of the spherical variables $\btheta$ and $\bphi$. 
Let $\alpha = [v_i, v_j, v_k]\in\K^2(\partial\M)$ be a boundary triangular face. We form the tetrahedron by the image of boundary face $f(\alpha)$ together with the origin $\0_3$ of $\R^3$, given by
$$
f(\tau_{\alpha}) = [\0_3, f(\alpha)] = [\0_3, f_i, f_j, f_k].
$$
The total image volume is then the sum of the volumes of these boundary tetrahedra
$$
\mathcal{V}(f) = \sum_{\alpha\in\K^2(\partial\M)} |f(\tau_{\alpha})|
= \frac{1}{6} \sum_{[v_i, v_j, v_k]\in\K^2(\partial\M)} f_i^\top (f_j \times f_k),
$$
as illustrated in Figure \ref{fig:1}. 
We denote the three coordinate vectors of its image $f(\alpha)$ by 
\[
\f_\alpha^s = (f_i^s, f_j^s, f_k^s)^\top, \quad s=1,2,3.
\]
Then, the signed volume of the tetrahedron $f(\tau_\alpha)$ can be formulated as 
\[
|f(\tau_\alpha)| = \frac{1}{6}\, f_i^\top (f_j \times f_k) = \frac{1}{6}\, \det 
\begin{bmatrix}
f_i^1  &  f_i^2  &  f_i^3\\
f_j^1  &  f_j^2  &  f_j^3\\
f_k^1  &  f_k^2  &  f_k^3\\
\end{bmatrix}  = 
\frac{1}{6} \, {\f^{1^\top}_\alpha} (\f^2_\alpha \times \f_\alpha^3),
\]
which coincides with the volume of tetrahedron $[\0_3, \f^1_\alpha, \f^2_\alpha, \f^3_\alpha]$.
Because moving $\f^1_\alpha$ parallel to the plane spanned by $\f^2_\alpha$ and $\f^3_\alpha$ does not change the volume, the gradient with respect to $\f^1_\alpha$ is normal to that plane.
Thus, the gradients with respect to $\f_\alpha^1$, $\f_\alpha^2$, and $\f_\alpha^3$ are
\begin{equation*}
\nabla_{\f_\alpha^1} |f(\tau_{\alpha})| = \frac{1}{6}(\f_\alpha^2 \times \f_\alpha^3), ~~
\nabla_{\f_\alpha^2} |f(\tau_{\alpha})| = \frac{1}{6}(\f_\alpha^3 \times \f_\alpha^1), ~~
\nabla_{\f_\alpha^3} |f(\tau_{\alpha})| = \frac{1}{6}(\f_\alpha^1 \times \f_\alpha^2).
\end{equation*}

\begin{figure}[tbp]
\centering
\includegraphics[height = 5.0cm]{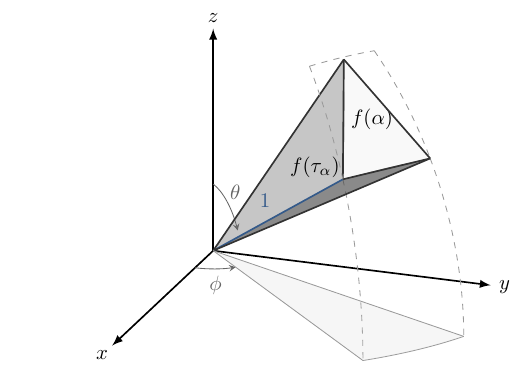}
\caption{An illustration for the tetrahedron $f(\tau_\alpha)$ formed by the boundary face $f(\alpha)$ and the origin $\0_3$ of $\R^3$.}
\label{fig:1}
\end{figure}

Now we represent $\f_\alpha^1$, $\f_\alpha^2$, $\f_\alpha^3$ in spherical coordinates $\btheta_\alpha = (\theta_i, \theta_j, \theta_k)^\top$ and $\bphi_\alpha = (\phi_i, \phi_j, \phi_k)^\top$ as
\begin{align*}
\f_\alpha^1 &= 
\begin{bmatrix}
\sin \theta_i \cos \phi_i \\
\sin \theta_j \cos \phi_j \\
\sin \theta_k \cos \phi_k \\
\end{bmatrix} 
= \begin{bmatrix}
\sin \theta_i \\
\sin \theta_j \\
\sin \theta_k \\
\end{bmatrix}
\odot
\begin{bmatrix}
\cos \phi_i \\
\cos \phi_j \\
\cos \phi_k \\
\end{bmatrix}
=: \sin\btheta_\alpha \odot \cos\bphi_\alpha,\\
\f_\alpha^2 &= 
\begin{bmatrix}
\sin \theta_i \sin \phi_i \\
\sin \theta_j \sin \phi_j \\
\sin \theta_k \sin \phi_k \\
\end{bmatrix} 
= \begin{bmatrix}
\sin \theta_i \\
\sin \theta_j \\
\sin \theta_k \\
\end{bmatrix}
\odot
\begin{bmatrix}
\sin \phi_i \\
\sin \phi_j \\
\sin \phi_k \\
\end{bmatrix}
=: \sin\btheta_\alpha \odot \sin\bphi_\alpha,\\
\f_\alpha^3 &= 
\begin{bmatrix}
\cos \theta_i \\ 
\cos \theta_j \\ 
\cos \theta_k 
\end{bmatrix} = \cos \btheta_\alpha,
\end{align*}
where $\odot$ denotes the Hadamard product of matrices or vectors. 
By applying the chain rule, the gradient of $|f(\tau_{\alpha})|$ with respect to $\btheta_\alpha$ and $\bphi_\alpha$ is formulated as
\begin{subequations} \label{eq:3.2}
\begin{align}
\nabla_{\btheta_\alpha} |f(\tau_{\alpha})| &= \frac{1}{6} \big( \cos\btheta_\alpha \odot \cos\bphi_\alpha \odot (\f^2_\alpha \times \f^3_\alpha) + \cos\btheta_\alpha \odot \sin\bphi_\alpha \odot (\f^3_\alpha \times \f^1_\alpha) \\
& \quad ~~ - \sin\btheta_\alpha \odot (\f^1_\alpha \times \f^2_\alpha) \big), \nonumber\\
\nabla_{\bphi_\alpha} |f(\tau_{\alpha})| &= \frac{1}{6} \big( \sin\btheta_\alpha \odot \cos\bphi_\alpha \odot (\f^3_\alpha \times \f^1_\alpha) - \sin\btheta_\alpha \odot \sin\bphi_\alpha \odot (\f^2_\alpha \times \f^3_\alpha) \big).
\end{align}
\end{subequations}
We denote the gradient of $\mathcal{V}(f)$ with respect to $\btheta$ and $\bphi$ as $\nabla_{\btheta} \mathcal{V}(f)$ and $\nabla_{\bphi} \mathcal{V}(f)$, respectively, which can be obtained by assembling associated tetrahedra using the formulae in \eqref{eq:3.2}.

On the other hand, from \cite[Theorem 3.5]{HuLL23}, the gradient of $E_\rV(f)$ with respect to $\f^s$ is formulated as 
\begin{equation} 
\nabla_{\f^s} E_\rV(f) = 2 \, L_\mathrm{V}(f) \,\f^s, ~~~ \text{for $s=1,2,3$}.
\label{eq:GradEs}
\end{equation}
Under a reordering with respect to $\I$ and $\B$, the gradient formula \eqref{eq:GradEs} is equivalent to
\begin{subequations}
\begin{align}
\nabla_{\f^s_\I} E_\rV(f) &= 2 \, \big( [L_\rV(f)]_{\I,\I} \f^s_\I + [L_\rV(f)]_{\I,\B} \f^s_\B \big), \label{eq:EV_grad_I}\\
\nabla_{\f^s_\B} E_\rV(f) &= 2 \, \big( [L_\rV(f)]_{\B,\I} \f_\I^s + [L_\rV(f)]_{\B,\B} \f_\B^s \big), ~~~ \text{for $s=1,2,3$}. 
\end{align}
\end{subequations}
Ultimately, we obtain the explicit gradient of $E_\rI(f)$ in  \eqref{eq:IEM} with respect to $\f_\I^s$, $\btheta$, and $\bphi$, which is formulated as 
\begin{subequations} \label{eq:Grad}
\begin{align}
\nabla_{\f_{\I}^s} E_\mathrm{I} (f) &=  \frac{\mathcal{V}({e})}{\mathcal{V}(f)} \nabla_{\f_{\I}^s} E_\mathrm{V}(f) \nonumber\\
&=  \frac{2 \mathcal{V}({e})}{\mathcal{V}(f)} \big( [L_\rV (f)]_{\I,\I} \f_{\I}^s + [L_\rV(f)]_{\I,\B} \f_{\B}^s \big), \, \textrm{for $s=1,2,3$,} 
\label{eq:EI_grad_I}
\end{align}
\begin{align}
\nabla_{\btheta} E_\rI (f) &= \nabla_{\btheta} \Big( \frac{\mathcal{V}({e})}{\mathcal{V}(f)} \Big) \, E_\mathrm{V}(f) + \frac{\mathcal{V}({e})}{\mathcal{V}(f)} \nabla_{\btheta} E_\mathrm{V}(f) - \nabla_{\btheta} \mathcal{V}(f) \nonumber\\
&= \frac{2 \mathcal{V}({e})}{\mathcal{V}(f)} \Big( \cos\btheta \odot \cos\bphi \odot \big( [L_\mathrm{V}(f)]_{\B,\I} \f^1_{\I} + [L_\mathrm{V}(f)]_{\B,\B} \f^1_{\B} \big) \nonumber\\
&\quad \quad \quad \quad ~~~~ + \cos\btheta \odot \sin\bphi \odot \big( [L_\mathrm{V}(f)]_{\B,\I} \f^2_{\I} + [L_\mathrm{V}(f)]_{\B,\B} \f^2_{\B} \big) \nonumber\\ 
&\quad \quad \quad \quad ~~~~ - \sin\btheta \odot \big( [L_\mathrm{V}(f)]_{\B,\I} \f^3_{\I} + [L_\mathrm{V}(f)]_{\B,\B} \f^3_{\B} \big) \Big) \nonumber\\
& \quad~  - \Big( 1 + \frac{\mathcal{V}({e}) \, E_\mathrm{V}(f) }{\mathcal{V}(f)^2}  \Big) \nabla_{\btheta}\mathcal{V}(f),
\end{align}
and
\begin{align}
\nabla_{\bphi} E_\rI (f) &= \nabla_{\bphi} \Big( \frac{\mathcal{V}({e})}{\mathcal{V}(f)} \Big) \, E_\mathrm{V}(f) + \frac{\mathcal{V}({e})}{\mathcal{V}(f)} \nabla_{\bphi} E_\mathrm{V}(f) - \nabla_{\bphi} \mathcal{V}(f) \nonumber\\
&=  \frac{2 \mathcal{V}({e})}{\mathcal{V}(f)} \Big( \sin\btheta \odot \cos\bphi \odot \big( [L_\rV(f)]_{\B,\I} \f^2_{\I} + [L_\rV(f)]_{\B,\B} \f^2_{\B} \big) \nonumber\\
&\quad \quad \quad \quad ~~~~ - \sin\btheta \odot \sin\bphi \odot \big( [L_\mathrm{V}(f)]_{\B,\I} \f^1_{\I} + [L_\mathrm{V}(f)]_{\B,\B} \f^1_{\B} \big) \Big) \nonumber\\
& \quad ~ - \Big( 1 + \frac{\mathcal{V}({e}) \, E_\mathrm{V}(f) }{\mathcal{V}(f)^2}  \Big) \nabla_{\bphi}\mathcal{V}(f). 
\end{align}
\end{subequations}

Using the explicit gradient formulas for $E_\mathrm{I}$ in \eqref{eq:Grad}, we develop a nonlinear CG method with appropriate preconditioning to solve the optimization problem \eqref{eq:IEM}, which is explained in the next section.

\section{Preconditioned nonlinear CG method of IEM}
\label{sec:4}

In this section, we present the algorithm that applies the preconditioned nonlinear CG method to solve the IEM problem in \eqref{eq:IEM}.

Since the target domain is the solid unit ball, we first globally scale the tetrahedral mesh by normalizing the vertices along the boundary's principal axes to make the mesh closer to a spherical shape. Let $\bv\in\mathbb{R}^{n\times 3}$ be the vertex matrix, $\bv_\B\in\mathbb{R}^{n_\B\times 3}$ be the boundary vertex matrix, and $\bar v_\B=\tfrac{1}{n_\B}\sum_{i=1}^{n_\B}\bv_i$ be its centroid. We compute the singular value decomposition (SVD) by
\[
(\bv_\B - \mathbf{1}\bar v_\B^\top) = U\Sigma R^\top,
\]
and then rotate and scale the whole mesh by
\begin{equation} \label{eq:AST}
\bv \leftarrow (\bv - \mathbf{1}\bar v_\B^\top) \, R \Sigma^{-1}.
\end{equation}
This alignment recenters the mesh and makes the boundary closer to a sphere.

For the initial mapping, we adopt the fixed-point method in \cite[Algorithm 4.4]{YuLL19}, which computes a volume-preserving map by solving a linear system at each iteration while keeping the boundary fixed. Although this approach does not guarantee global convergence, it typically achieves a rapid energy decrease in the first few iterations. After about 10 iterations, however, the improvement becomes marginal (see Section~\ref{sec:6.1}). We therefore run this fixed-point method for 15 iterations and use the resulting map as the initial guess for our preconditioned nonlinear CG algorithm.

At the $k$th iteration, we represent the mapping $f^{(k)}$ by
\begin{equation} \label{eq:sphere f}
\f^{(k)} :=
\begin{bmatrix}
{\f_\I^1}^{(k)} \\
{\f_\I^2}^{(k)} \\
{\f_\I^3}^{(k)} \\
\btheta^{(k)} \\
\bphi^{(k)}
\end{bmatrix}.
\end{equation} 
The corresponding gradient vector is denoted as
\begin{equation} \label{eq:hat g}
\g^{(k)} 
:= \nabla E_\mathrm{I}(\f^{(k)})
= \begin{bmatrix}
\nabla_{\f_\I^1} E_\mathrm{I}(f^{(k)}) \\
\nabla_{\f_\I^2} E_\mathrm{I}(f^{(k)}) \\
\nabla_{\f_\I^3} E_\mathrm{I}(f^{(k)}) \\
\nabla_{\btheta} E_\mathrm{I}(f^{(k)}) \\
\nabla_{\bphi}   E_\mathrm{I}(f^{(k)})
\end{bmatrix}
=: \begin{bmatrix}
{\g_\I^1}^{(k)} \\
{\g_\I^2}^{(k)} \\
{\g_\I^3}^{(k)} \\
\g_{\btheta}^{(k)} \\
\g_{\bphi}^{(k)}
\end{bmatrix},
\end{equation}
which is explicitly formulated in \eqref{eq:Grad}.

Starting from $\f^{(0)}$, we update it iteratively by
\begin{subequations} \label{eq:PCG}
\begin{equation}
\f^{(k+1)} = \f^{(k)} + \alpha_k \p^{(k)}, 
\end{equation}
where $\alpha_k$ is chosen by a line search satisfying the Wolfe conditions, and ${\p}^{(k)}$ is the preconditioned CG direction at the $k$th step, given by
\begin{equation} \label{eq:PCGb}
\p^{(k)} :=
\begin{bmatrix}
{\p_\I^1}^{(k)}\\
{\p_\I^2}^{(k)}\\
{\p_\I^3}^{(k)}\\
\p_{\btheta}^{(k)}\\
\p_{\bphi}^{(k)}
\end{bmatrix} 
= - M^{-1} \g^{(k)} + \beta_k \p^{(k-1)}, \qquad 
\p^{(0)} = - M^{-1} \g^{(0)}
\end{equation}
with the corrective scalar computed in the $M^{-1}$-inner product,
\begin{equation} \label{eq:beta}
\beta_k = \frac{ {\g^{(k)}}{}^\top M^{-1} \g^{(k)} }{ {\g^{(k-1)}}^\top M^{-1} \g^{(k-1)}} = \frac{\| {\g^{(k)}} \|^2_{M^{-1}}}{\| {\g^{(k-1)}} \|^2_{M^{-1}}},
\end{equation}
\end{subequations}
where $M$ is a suitable preconditioner.

A good preconditioner is a symmetric positive definite matrix that approximates the Hessian $\nabla^2 E_\rI$. When the image volume $\mathcal{V}(f)$ varies mildly, $E_\rI(f)$ is well approximated by $E_\mathrm{V}(f)$, and applying the product rule to the gradient $\nabla_{\f^s} E_\mathrm{V}$ in \eqref{eq:GradEs} gives
\begin{equation}
\nabla^2_{\f^s} E_\mathrm{V} (\f) = 2\,L_\rV(\f)  + 2 \,\mathrm{D}_{\f^s} L_\rV(\f) \, \f^s, \quad s = 1,2,3,
\label{eq:Hessian}
\end{equation}
where the third-order tensor $\mathrm{D}_{\f^s} L_\rV(\f) \in \mathbb{R}^{n\times n\times n}$ is defined by
$$
[\mathrm{D}_{\f^s} L_\rV(\f)]_{i,j,\ell} = \frac{\partial}{\partial f_\ell^s}[L_\rV(\f)]_{i,j},
$$
and the matrix $\mathrm{D}_{\mathbf f^s}L_{\mathrm V}(\mathbf f)\,\mathbf f^s\in\mathbb{R}^{n\times n}$ is given by 
$$
[\mathrm{D}_{\f^s} L_\rV(\f) \,\f^s]_{i,\ell}
= \sum_{j=1}^n [\mathrm{D}_{\f^s} L_\rV(\f)]_{i,j,\ell} f_j^s
=\sum_{j=1}^n \frac{\partial}{\partial f_\ell^s}[L_\rV(\f)]_{i,j} f_j^s.
$$
Since the principal submatrices of $L_\rV(f)$ are typically symmetric positive definite, we keep the first term in \eqref{eq:Hessian} and adopt the block-diagonal preconditioner
\begin{equation}
M = \begin{bmatrix}
I_3 \otimes [L_\rV(f^{(0)})]_{\I,\I} \\
& I_2 \otimes [L_\rV(f^{(0)})]_{\B,\B}
\end{bmatrix}.
\label{eq:precond}
\end{equation}
Since $M$ is fixed, in practice, we compute the preordered Cholesky decomposition so that the linear systems involving $[L_\rV(f^{(0)})]_{\I, \I}$ and $[L_\rV(f^{(0)})]_{\B, \B}$ can be efficiently solved using forward and backward substitution in each iteration.

For the CG correction parameter $\beta_k$, we tested several standard choices (see Section~\ref{sec:6.2}) and found that they produced broadly similar numerical behavior on our IEM problem. Moreover, the preconditioned variants consistently outperformed the unpreconditioned ones. We therefore adopt the preconditioned Fletcher--Reeves formula \eqref{eq:beta}, for which global convergence is established in Section~\ref{sec:5}.

The step length from an exact line search satisfies
\[
\alpha_* = \argmin \varphi(\alpha) \equiv E_{\rI} \big(\f^{(k)} +\alpha \p^{(k)} \big).
\]
Because $E_\rI$ is nonlinear, we approximate $\alpha_*$ by quadratic and cubic interpolation \cite[Chapter~3.5]{NoWr06}, using the previous step length as the initial trial step, denoted as $\alpha_0$. 

In the quadratic model, we construct $\varphi_q$ to match the value and derivative at $0$ and the value at $\alpha_0$:
\begin{equation}
\begin{cases}
\varphi_q(0) = \varphi(0) = E_\mathrm{I}\big(\f^{(k)}\big), \\
\varphi_q'(0) = \varphi'(0) 
= {{\p}^{(k)}}^\top \nabla E_\mathrm{I}\big(\f^{(k)}\big), \\
\varphi_q(\alpha_{0}) = \varphi(\alpha_{0}) = E_\rI\big(\f^{(k)} + \alpha_{0} {\p}^{(k)}\big).
\end{cases}
\label{eq:quad_interp}
\end{equation}
This gives
\begin{equation} 
\varphi_q(\alpha) \equiv a \, \alpha^2 + \varphi'(0) \, \alpha + \varphi(0), \quad a= \frac{\varphi(\alpha_0) - \varphi(0) - \alpha_0 \varphi'(0)  }{\alpha_0^2}
\label{eq:quad_model}
\end{equation}
and the quadratic minimizer
\begin{equation} 
\alpha_1 = \argmin \varphi_q(\alpha) = -\frac{\varphi'(0)}{2a}.
\label{eq:quad_step}
\end{equation}

If $\alpha_1$ does not provide a sufficient decrease that satisfies
\begin{equation}
\varphi(\alpha_1) < \ell(\alpha_1) \equiv \varphi(0) + c_1 \alpha_1 \varphi'(0),
\label{eq:decreas_cond}
\end{equation}
where $c_1 = \tfrac{1}{2} - \varepsilon$ with $\varepsilon$ being a positive number close to zero, we switch to a cubic model 
\begin{equation} 
\varphi_c(\alpha)  = b \, \alpha^3 + c \, \alpha^2 + \varphi'(0) \, \alpha + \varphi(0)
\label{eq:cubic_model}
\end{equation}
that satisfies \eqref{eq:quad_interp} together with $\varphi_c(\alpha_1) = \varphi(\alpha_1)$. The coefficients can be obtained by
\begin{equation}
\begin{bmatrix}
b\\
c
\end{bmatrix} = \frac{1}{\alpha_0^2 \alpha_1^2 (\alpha_1 - \alpha_0)}
\begin{bmatrix}
    \alpha_0^2  &  - \alpha_1^2\\
    -\alpha_0^3 &  \alpha_1^3
\end{bmatrix}
\begin{bmatrix}
\varphi(\alpha_1) - \varphi(0) - \varphi'(0) \alpha_1\\
\varphi(\alpha_0) - \varphi(0) - \varphi'(0) \alpha_0\\
\end{bmatrix}.  
\end{equation}
and the cubic minimizer is
\begin{equation}
\alpha_2 = \argmin \varphi_c(\alpha) =  \frac{-c + \sqrt{c^2 - 3 \,b \varphi'(0)}}{3 b}.
\label{eq:cubic_step}
\end{equation}
The computational procedure is outlined in detail in Algorithm \ref{alg:IEM}.

We note that the requirement of $c_1$ in \eqref{eq:decreas_cond} coincides with the condition \eqref{eq:Wolfe1}, which is used in our global convergence analysis in Section~\ref{sec:5}. In practice, taking the quadratic step and scaling it by $0.9$ satisfies this condition in nearly all our tests. A numerical illustration is given in Section~\ref{sec:6.1}.

\begin{algorithm}
\caption{Preconditioned Nonlinear CG Method for IEM}
\label{alg:IEM}
\begin{algorithmic}[1]
\Require A simply connected tetrahedral mesh $\mathcal{M}$ and a tolerance $\epsilon$.
\Ensure An approximate volume-preserving simplicial mapping $f$.
\State Compute the SVD and update vertices by \eqref{eq:AST}.
\State Let $n$ be the number of elements in $\K^0(\M)$.
\State Let $\B = \lbrace s \,|\, v_s \in \partial \mathcal{M} \rbrace$ and $\I = \lbrace 1, \cdots, n \rbrace \setminus \B$.
\State Compute a boundary mapping $\f_\B^s$, for $s = 1,2,3$, using \cite[Algorithm 4.3]{YuLL19}.
\State Let $L \gets L_\mathrm{V}(e)$ as \eqref{eq:VSLaplacian}.
\For {$k = 1, \ldots, 15$}
    \State Solve the linear system $L_{\I,\I} \f_\I^s  = -L_{\I,\B} \f_\B^s$, for $s=1,2,3$.
    \State Update $L \gets L_\mathrm{V}(f)$ as \eqref{eq:VSLaplacian}.
\EndFor
\State Let the preconditioners $M_\I \gets L_{\I,\I}$ and $M_\B \gets L_{\B,\B}$.
\State Perform Cholesky decompositions of $M_\I$ and $M_\B$.
\State Compute the spherical coordinates $[\btheta, \bphi]$ by \eqref{eq:sphere_coor_inv}.
\State Compute the energy value $E = E_\mathrm{I}(f)$ as \eqref{eq:Ea}.
\State Let $\delta \leftarrow \infty$.
\While{$\delta > \epsilon$}
    \State Let $E_0 \leftarrow E$.
    \State Compute gradients $\g_{\I}^s$, $\g_{\btheta}$, $\g_{\bphi}$ for $s=1,2,3$ by \eqref{eq:Grad}.
    \State Solve $M_\I \h_\I^s = \g_\I^s$, $M_\B \h_\btheta = \g_\btheta$ and $M_\B \h_\bphi = \g_\bphi$, for $s=1,2,3$.
    \If {$\delta = \infty$}
        \State Compute $\lambda \leftarrow \sum_{s=1}^3 ({\g_\I^s}^\top \h_\I^s) + \g_\btheta^\top \h_\btheta + \g_\bphi^\top \h_\bphi$.
        \State Update $\p_\I^s \leftarrow - \h_\I^s$, $\p_\btheta \leftarrow - \h_\btheta$, and $\p_\bphi \leftarrow - \h_\bphi$, for $s=1,2,3$.
    \Else
        \State Let $\lambda_{0} \leftarrow \lambda$.
        \State Compute $\lambda \leftarrow \sum_{s=1}^3 ({\g_\I^s}^\top \h_\I^s) + \g_\btheta^\top \h_\btheta + \g_\bphi^\top \h_\bphi$.
        \State Let $\beta \leftarrow \lambda / \lambda_{0}.$
        \State Update $\p_\I^s \leftarrow - \h_\I^s + \beta \p_\I^s$, $\p_\btheta \leftarrow - \h_\btheta + \beta \p_\btheta$, and $\p_\bphi \leftarrow - \h_\bphi + \beta \p_\bphi$, for $s=1,2,3$.
    \EndIf
    \State Compute a step length $\alpha$ by \eqref{eq:quad_step} or \eqref{eq:cubic_step}.
    \State Update $\f_\I^s \leftarrow \f_\I^s + \alpha \p_\I^s$, $\btheta \leftarrow \btheta + \alpha \p_\btheta$, and $\bphi \leftarrow \bphi + \alpha \p_\bphi$, for $s=1,2,3$.
    \State Update $\f_\B^s$ by the spherical-coordinate formula \eqref{eq:sphere_coor}, for $s=1,2,3$.
    \State Update $L\gets L_\mathrm{V}(f)$ as \eqref{eq:VSLaplacian}.
    \State Compute the energy value $E = E_\mathrm{I}(f)$ as \eqref{eq:Ea}.
    \State Update $\delta \leftarrow E_0 - E$.
\EndWhile
\end{algorithmic}
\end{algorithm}

\section{Global convergence analysis}
\label{sec:5}
In this section, we provide a proof of the global convergence of the proposed Algorithm~\ref{alg:IEM}. The main theorem (Theorem~\ref{thm:2}) generalizes the convergence analysis of the nonlinear CG method introduced in \cite[Chapter 5]{NoWr06}. Similar techniques have been applied to the convergence analysis of authalic energy minimization for disk area-preserving parameterizations \cite{LiYu24}. One of the main differences is that we change the definition of the denominator of $\cos\theta_M^{(k)}$ in the modified Zoutendijk's condition (Lemma~\ref{lma:Zou}) so that the proof of Theorem~\ref{thm:2} becomes simpler.

Theoretical guarantees for the global convergence of Algorithm \ref{alg:IEM} are established, assuming that each step length $\alpha_k$ satisfies the strong Wolfe conditions 
\begin{subequations} \label{eq:Wolfe}
\begin{align}
E_\rI(\f^{(k+1)}) - E_\rI(\f^{(k)})& \leq  c_1 \, \alpha_k {\g^{(k)}}^{\top} \p^{(k)}, \label{eq:Wolfe1} \\
|{\g^{(k+1)}}^{\top}  \p^{(k)}| & \leq c_2 \, | {\g^{(k)}}^{\top} \p^{(k)}|, \label{eq:Wolfe2}
\end{align}
\end{subequations}
with $0 < c_1 < c_2 < \frac{1}{2}$. 
The existence of each desirable value of $\alpha_k$ is ensured by Lemma~\ref{lma:Wolfe}, assuming that $\p^{(k)}$ is a descent direction of $E_\mathrm{I}$ at $\f^{(k)}$.

In the following lemma, we establish the Lipschitz continuity of both the objective functional and its gradient, which are crucial for proving the convergence of the proposed method.

\begin{lemma}\label{lma:Lipschitz}
The isovolumetric energy $E_\rI$ in \eqref{eq:Ea} and its gradient $\nabla E_\rI$ are both Lipschitz continuous with respect to $\f$ in \eqref{eq:sphere f}, provided that the image volume is bounded away from zero, i.e., $\mathcal{V}(\f)\geq v_*$ for some constant $v_*>0$.
\end{lemma}

\begin{proof}
Recall that $\f$ represents the vertex-image vector of $f$, with boundary vertices written in spherical coordinates. We consider \(\mathbf f\) in a compact feasible set 
$$
D = [-1,1]^{3 n_\I} \times [0, \pi]^{n_\B} \times [-\pi, \pi]^{n_\B} \cap \{ \f \mid \mathcal{V}(\f) \geq v_* \},
$$
where $n_\I$ and $n_\B$ are the numbers of interior and boundary vertices.
Since the functional $E_\rI$ in \eqref{eq:Ea} is built from the volume $|f(\tau)|$ and $\mathcal{V}(\f) \geq v_* > 0$, it is smooth with respect to the vertex image $f_i$ on $[-1,1]^{3}$. By the chain rule, $E_\rI$ is smooth with respect to $\f$ on $D$ as well.

Let $\u,\v\in D$ such that the line segment between $\u, \v$ lies in $D$ and define $\phi(t) = E_\rI((1-t)\u + t\v)$. Since $\phi$ is smooth, by the mean value theorem, there exists $\xi_1 \in (0,1)$ such that 
\[
E_\rI(\v) - E_\rI(\u) = \phi(1) - \phi(0) = \phi'(\xi_1) = \nabla E_\rI(\w_1)^\top (\v - \u),
\]
where $\w_1 = (1-\xi_1) \u + \xi_1 \v$. Hence,
\[
|E_\rI(\v) - E_\rI(\u)| \leq \| \nabla E_\rI(\w_1)\|_2\,  \|\v - \u\|_2 \leq \sup_D \|\nabla E_\rI(\f)\|_2 \|\v - \u\|_2.
\]
Since \(D\) is compact and \(\nabla E_{\mathrm I}\) is continuous, by the extreme value theorem, the supremum is finite, so \(E_{\mathrm I}\) is Lipschitz on \(D\).

On the other hand, to show the Lipschitz continuity of $\nabla E_\rI$, we define $\varphi(t) = (\nabla E_\rI(\v) - \nabla E_\rI(\u))^\top \nabla E_\rI((1-t)\u + t\v)$ for $\u,\v\in D$ such that the line segment between $\u, \v$ lies in $D$. Due to the smoothness of $\nabla E_\rI$, the mean value theorem guarantees the existence of a value $\xi_2 \in(0,1)$ such that 
\begin{align*}
\|\nabla E_\rI(\v) - \nabla E_\rI(\u)\|_2^2 
&= \varphi(1) - \varphi(0) = \varphi'(\xi_2) \\
&= (\nabla E_\rI(\v) - \nabla E_\rI(\u))^\top \nabla^2 E_\rI(\w_2) (\v - \u),
\end{align*}
where $\w_2 = (1-\xi_2) \u + \xi_2 \v$. Hence,
\begin{align*}
\|\nabla E_\rI(\v) - \nabla E_\rI(\u)\|_2^2 
&\leq  \|\nabla E_\rI(\v) - \nabla E_\rI(\u) \|_2  \,\|\nabla^2 E_\rI(\w_2)\|_2  \,\|\v - \u\|_2\\
&\leq \sup_D \|\nabla^2 E_\rI(\f)\|_2 \, \|\nabla E_\rI(\v) - \nabla E_\rI(\u) \|_2 \|\v - \u\|_2.
\end{align*}
Since \(D\) is compact and \(\nabla^2 E_{\mathrm I}\) is continuous, by the extreme value theorem, the supremum is finite, so \(\nabla E_{\mathrm I}\) is Lipschitz on \(D\).
\end{proof}

To ensure that the iterative method converges to a critical point of $E_\rI$, we need a descent direction and a step length that together produce a sufficient decrease of $E_\rI$ at each iteration. 
Since $E_\rI$ is bounded below, whenever $\p^{(k)}$ is a descent direction, there exists a step length satisfying the strong Wolfe conditions \eqref{eq:Wolfe}, as stated in the following lemma.

\begin{lemma}[{\cite[Lemma 3.1]{NoWr06}}]
\label{lma:Wolfe}
Let $\p^{(k)}$ be a descent direction at $\f^{(k)}$. For $0 < c_1 < c_2 < 1$, there exist intervals of step lengths satisfying the strong Wolfe conditions \eqref{eq:Wolfe}.
\end{lemma}

For $k = 0$, the initial direction $\p^{(0)} = - M^{-1} \g^{(0)}$ is a descent direction if $M^{-1}$ is symmetric and positive definite, which guarantees the existence of a step length satisfying the strong Wolfe conditions. 
To ensure that $\p^{(k)}$ remains a descent direction for $k \geq 1$, the parameter $c_2$ in \eqref{eq:Wolfe2} should be chosen from the interval $(0,\frac{1}{2})$, as shown in the following lemma, which is a variant of \cite[Lemma 5.6]{NoWr06}.

\begin{lemma}
\label{lma:descent} 
Let Algorithm~\ref{alg:IEM} be implemented with step length $\alpha_k$ satisfying the strong Wolfe conditions \eqref{eq:Wolfe} with a parameter $ c_2 \in (0, \frac{1}{2})$. 
Then, for every $k$, $\p^{(k)}$ is a descent direction and satisfies
\begin{equation} \label{eq:descent_cond}
-\frac{1}{1-c_2} \leq \frac{{\g^{(k)}}^{\top} \p^{(k)}}{\| \g^{(k)} \|^2_{M^{-1}}} \leq \frac{2c_2 - 1}{1 - c_2},
\end{equation}
provided that $M$ is a symmetric positive definite matrix. Here, $\|\cdot\|_{M}$ denotes the $M$-norm, defined as $\|\v\|_{M} = \sqrt{\v^\top M\v}$.
\end{lemma}

\begin{proof}
First, we note that $\varphi(c_2) = \frac{2c_2 - 1}{1 - c_2}$ is monotone increasing and $\varphi(\frac{1}{2}) = 0$. Hence, if $c_2 \in (0, \frac{1}{2})$, the right-hand side of \eqref{eq:descent_cond} is negative, so $\p^{(k)}$ is a descent direction once \eqref{eq:descent_cond} is established.

We prove \eqref{eq:descent_cond} by induction on $k$. 
For $k = 0$, since $\p^{(0)} = - M^{-1} \g^{(0)}$, we have 
$$\frac{{\g^{(0)}}^{\top} \p^{(0)}}{\| {\g^{(0)}} \|^2_{M^{-1}}} = -1,$$ 
which satisfies \eqref{eq:descent_cond}. 
Assume the induction hypothesis \eqref{eq:descent_cond} holds for $k=n$. 
Then, for $k = n + 1$, from \eqref{eq:PCGb} and \eqref{eq:beta}, we have
\begin{align}
\frac{{\g^{(n+1)}}^{\top} \p^{(n+1)}}{\| {\g^{(n+1)}} \|^2_{M^{-1}}} 
&= -1 + \beta_{n+1} \frac{{\g^{(n+1)}}^{\top} \p^{(n)}}{\| {\g^{(n+1)}} \|^2_{M^{-1}}} 
= -1 + \frac{{\g^{(n+1)}}^{\top} \p^{(n)}}{\| {\g^{(n)}} \|^2_{M^{-1}}}. \label{eq:4.13} 
\end{align}
By the induction hypothesis, $\p^{(k)}$ is a descent direction. Hence, by Lemma~\ref{lma:Wolfe}, there exists a step length satisfying the second Wolfe condition \eqref{eq:Wolfe2}, given by
\begin{align}
c_2 \, {\g^{(n)}}^{\top} \p^{(n)} \leq {\g^{(n+1)}}^{\top} \p^{(n)} \leq -c_2 \, {\g^{(n)}}^{\top} \p^{(n)}. \label{eq:4.15} 
\end{align}
By combining \eqref{eq:4.13} and \eqref{eq:4.15}, we obtain
\begin{align*}
-1 + c_2\frac{{\g^{(n)}}^\top \p^{(n)}}{\| {\g^{(n)}} \|^2_{M^{-1}}}
&\leq \frac{{\g^{(n+1)}}^\top \p^{(n+1)}}{\| {\g^{(n+1)}} \|^2_{M^{-1}}} 
\leq -1 - c_2\frac{{\g^{(n)}}^{\top} \p^{(n)}}{\| {\g^{(n)}} \|^2_{M^{-1}}}.
\end{align*}
Then, from the induction hypothesis \eqref{eq:descent_cond} with $k = n$ again, we have
\begin{equation*}
-\frac{1}{1-c_2} = -1 - \frac{c_2}{1-c_2} \leq \frac{{\g^{(n+1)}}^{\top} \p^{(n+1)}}{\| \g^{(n+1)} \|^2_{M^{-1}}} \leq -1 + \frac{c_2}{1-c_2} = \frac{2c_2 - 1}{1 - c_2}, 
\end{equation*}
i.e., \eqref{eq:descent_cond} holds for $k=n+1$.
\end{proof}

This sufficient descent property provided by Lemma~\ref{lma:descent}, together with the strong Wolfe line search, leads to a global bound on the sum of the squared gradient components along the search directions. This is stated in the following lemma, which is a variant of Zoutendijk’s condition.

\begin{lemma}[modified Zoutendijk's condition] \label{lma:Zou}
Let $M$ be a symmetric positive definite matrix.
Suppose Algorithm~\ref{alg:IEM} is implemented with the step length $\alpha_k$ satisfying the strong Wolfe conditions \eqref{eq:Wolfe} with $0<c_1<c_2<\frac{1}{2}$, and $\mathcal{V}(\f) \geq v_*$ on each segment $[\f^{(k)}, \f^{(k+1)}]$.
Then, 
\begin{subequations} \label{eq:Zou}
\begin{equation}
\sum_{k=0}^{\infty} \cos^2 \theta_{M}^{(k)} \|\g^{(k)} \|_{M^{-1}}^2 < \infty,
\end{equation}
where $\cos\theta^{(k)}_M$ is defined as
\begin{equation} \label{eq:cos}
\cos\theta_{M}^{(k)} = \frac{-{\g^{(k)}}^\top \p^{(k)}}{\| \g^{(k)} \|_{M^{-1}} \| \p^{(k)} \|_{M}},
\end{equation}
\end{subequations}
\end{lemma}
\begin{proof} 
Recall that the second Wolfe condition \eqref{eq:Wolfe2} states that
\begin{equation} \label{eq:lma4.4_1}
{\g^{(k+1)}}^{\top} \p^{(k)} \geq c_2\, {\g^{(k)}}^{\top} \p^{(k)}.
\end{equation}
Subtracting ${\g^{(k)}}^{\top} \p^{(k)}$ from both sides of \eqref{eq:lma4.4_1} yields
\begin{equation} \label{eq:lma4.4_2}
\big({\g^{(k+1)}} - {\g^{(k)}}\big)^{\top} \p^{(k)} \geq (c_2-1) \, {\g^{(k)}}^{\top} \p^{(k)}.
\end{equation}
Since $\nabla E_\rI$ is Lipschitz continuous, \eqref{eq:lma4.4_2} implies
\begin{equation} \label{eq:lma4.4_3}
(c_2-1) \, {\g^{(k)}}^{\top} \p^{(k)}
\leq \| {\g^{(k+1)}} - {\g^{(k)}} \|_2 \,\|\p^{(k)}\|_2 
\leq \eta \, \| {\f}^{(k+1)} - {\f}^{(k)} \|_2 \,\| \p^{(k)}\|_2.
\end{equation}
for some $\eta > 0$.
The values in \eqref{eq:lma4.4_3} can be estimated in $M$-norm as
\begin{equation} \label{eq:lma4.4_4}
\| {\f}^{(k+1)} - {\f}^{(k)} \|_2
= \| M^{-1/2} ({\f}^{(k+1)} - {\f}^{(k)}) \|_{M}
\leq \|M^{-1/2}\|_{M} \| {\f}^{(k+1)} - {\f}^{(k)} \|_{M},
\end{equation}
and
\begin{equation} \label{eq:lma4.4_5}
\| \p^{(k)}\|_2 = \| M^{-1/2} \p^{(k)}\|_{M}
\leq \|M^{-1/2}\|_{M} \|\p^{(k)}\|_{M}.
\end{equation}
From \eqref{eq:lma4.4_3}--\eqref{eq:lma4.4_5}, we have
\begin{equation} \label{eq:lma4.4_alpha_k}
\alpha_k = \frac{\| {\f}^{(k+1)} - {\f}^{(k)} \|_{M}}{\|\p^{(k)}\|_{M}}
= \frac{\| {\f}^{(k+1)} - {\f}^{(k)} \|_{M} \|\p^{(k)}\|_{M}}{\|\p^{(k)}\|_{M}^2}
\geq \frac{c_2-1}{ \eta \|M^{-1/2}\|^2_{M} } \frac{{\g^{(k)}}^{\top} \p^{(k)}}{\| \p^{(k)}\|^2_{M}}.
\end{equation}
By substituting \eqref{eq:lma4.4_alpha_k} into the first Wolfe condition \eqref{eq:Wolfe1}, we obtain
\begin{align}
E_\mathrm{I}({\f}^{(k+1)}) 
&\leq E_\mathrm{I}(\f^{(k)}) - c_1 \frac{1 -c_2}{ \eta \, \|M^{-1/2}\|^2_{M}} \frac{{\g^{(k)}}^{\top} \p^{(k)}}{\| \p^{(k)}\|^2_{M}} {\g^{(k)}}^\top \p^{(k)} \nonumber\\
&\leq E_\mathrm{I}({\f}^{(k)}) - \widetilde\eta  \frac{({\g^{(k)}}^{\top} \p^{(k)})^2}{\| \p^{(k)}\|^2_{M}} \nonumber\\
&= E_\mathrm{I}({\f}^{(k)}) - \widetilde \eta\, \cos^2\theta_M^{(k)} \|{\g^{(k)}}\|_{M^{-1}}^2, \label{eq:4.12}
\end{align}
where $\widetilde\eta =\frac{c_1 (1-c_2)}{ \eta \, \|M^{-1/2}\|^2_{M}} > 0$ and $\cos\theta_M^{(k)}$ is given by \eqref{eq:cos}. 
By applying the recursive inequality of \eqref{eq:4.12}, we obtain
\begin{equation} \label{eq:3.15}
\sum_{j = 0}^{k} \cos^2 \theta_{M}^{(j)} \| {\g^{(j)}} \|_{M^{-1}}^2 \leq \frac{1}{\widetilde\eta}\big( E_\mathrm{I}({\f}^{(0)}) - E_\mathrm{I}({\f}^{(k+1)}) \big).
\end{equation}
Since $E_\mathrm{I}({\f}) \geq 0$, the value $E_\mathrm{I}({\f}^{(0)}) - E_\mathrm{I}({\f}^{(k+1)})$ in \eqref{eq:3.15} is bounded. Therefore,
$$
\lim_{k\to\infty} \sum_{j = 0}^{k} \cos^2 \theta_{M}^{(j)} \| {\g^{(j)}} \|_{M^{-1}}^2 \leq \sup_{\f} \frac{1}{\widetilde\eta}\big( E_\mathrm{I}({\f}^{(0)}) - E_\mathrm{I}({\f}) \big) < \infty,
$$
which is equivalent to \eqref{eq:Zou}.
\end{proof}

Lemma~\ref{lma:Zou} shows that, along the iterations, either the gradient norm eventually becomes small or the search direction becomes nearly orthogonal to $-\g^{(k)}$. While, as $\p^{(k)} = - M^{-1} {\g^{(k)}}$, the global convergence follows directly because $\cos \theta_M^{(k)} = 1$ for all $k$, in the nonlinear CG method, the correction term may make $\p^{(k)}$ deviate from $-M^{-1} \g^{(k)}$. Despite this, the bound on $-{\g^{(k)^\top}} \p^{(k)}$ provided by Lemma~\ref{lma:descent} makes $\cos\theta_M^{(k)}$ under control. This allows us to establish global convergence, as stated in the following theorem.

\begin{theorem} \label{thm:2}
Assume $M$ is a symmetric positive definite matrix. Let Algorithm~\ref{alg:IEM} be implemented with $\alpha_k$ satisfying the strong Wolfe conditions \eqref{eq:Wolfe} with $0<c_1<c_2<\frac{1}{2}$, and $\mathcal{V}(\f) \geq v_*$ on each segment $[\f^{(k)}, \f^{(k+1)}]$. 
Then,
\begin{equation} \label{eq:PCG_Convergence}
\liminf_{k \to \infty}{\| \g^{(k)} \|_{M^{-1}} } = 0.
\end{equation}
\end{theorem}

\begin{proof}
We establish \eqref{eq:PCG_Convergence} using a proof by contradiction.
Suppose, to the contrary, that there exists a natural number $n_0$ and a positive real number $\gamma > 0$ such that $\| \g^{(k)} \|_{M^{-1}} \geq \gamma$, for every natural number $k>n_0$. 
Moreover, if Algorithm~\ref{alg:IEM} does not terminate after finitely many iterations, we have $\| \g^{(k)} \|_{M^{-1}}  > 0 $ for all $k$. Hence, there exists a constant $\varepsilon > 0$ such that 
\begin{equation}
\| \g^{(k)} \|_{M^{-1}} \geq \varepsilon, \qquad \mbox{for all }k.
\label{eq:1}
\end{equation}

Since each step length $\alpha_k$ satisfies the second Wolfe condition \eqref{eq:Wolfe2} with $c_2 \in (0, \tfrac{1}{2})$, Lemma~\ref{lma:descent} implies that $\p^{(k)}$ is a descent direction and 
\begin{equation}
\frac{1-2c_2}{1-c_2}  \leq \frac{-{\g^{(k)}}^{\top} \p^{(k)}}{\| \g^{(k)} \|^2_{M^{-1}}} \leq \frac{1}{1 - c_2}
\label{eq:1.5}
\end{equation}
Substituting $\cos \theta_M^{(k)}$ in \eqref{eq:cos} into \eqref{eq:1.5} yields
\begin{equation} \label{eq:2}
\frac{1-2c_2}{1-c_2} \frac{\|\g^{(k)}\|_{M^{-1}}}{\|\p^{(k)}\|_M} \leq \cos \theta_M^{(k)} \leq  \frac{1}{1 - c_2} \frac{\|\g^{(k)}\|_{M^{-1}}}{\|\p^{(k)}\|_M}. 
\end{equation}
Since $\p^{(k)}$ is a descent direction, by Lemma~\ref{lma:Zou} and left inequality \eqref{eq:2}, we obtain
\begin{align}
\sum_{k=0}^{\infty}  \frac{\| \g^{(k)} \|^4_{M^{-1}}}{\| \p^{(k)} \|^2_{M}} < \infty.  \label{eq:3}
\end{align}
Then, substituting the assumption \eqref{eq:1} into \eqref{eq:3} gives
\begin{equation}
\sum_{k=0}^{\infty}  \frac{1}{\| \p^{(k)} \|^2_{M}} < \infty. \label{eq:4}
\end{equation}

The $M$-norm of $\p^{(k)} = -M^{-1} \g^{(k)} + \beta_k \p^{(k-1)}$ can be estimated by
\begin{align}
\| \p^{(k)} \|^2_{M}
&=  \| M^{-1} \g^{(k)} \|_{M}^2 - 2 \beta_k (M^{-1} \g^{(k)} )^\top (M \p^{(k-1)}) + \beta^2_k \, \|\p^{(k-1)}\|_{M}^2 \nonumber\\
& \leq \| \g^{(k)} \|^2_{M^{-1}} + 2\beta_k \, |{\g^{(k)}}^{\top} \p^{(k-1)} | + \beta^2_k \, \|\p^{(k-1)}\|_{M}^2. \label{eq:p_k_M2}
\end{align}
The second Wolfe condition \eqref{eq:Wolfe2} and the right inequality of \eqref{eq:1.5} imply
\begin{equation}
| {\g^{(k)}}^{\top} \p^{(k-1)} | \leq -c_2 ~{\g^{(k-1)}}^{\top} \p^{(k-1)} \leq \frac{c_2}{1-c_2} \| {\g^{(k-1)}} \|^2_{M^{-1}}.  \label{eq:5}
\end{equation}
Thus, \eqref{eq:p_k_M2}, \eqref{eq:5}, and \eqref{eq:beta} imply
\begin{align}
\| \p^{(k)} \|^2_{M}
&\leq \| {\g^{(k)}} \|^2_{M^{-1}} + \frac{2c_2}{1-c_2} \beta_k  \| {\g^{(k-1)}} \|^2_{M^{-1}} + \beta^2_k \, \|\p^{(k-1)}\|_{M}^2 \nonumber\\
&= \| {\g^{(k)}} \|^2_{M^{-1}} + \frac{2c_2}{1-c_2} \| {\g^{(k)}}\|^2_{M^{-1}}  + \beta^2_k \, \|\p^{(k-1)}\|_{M}^2 \nonumber\\
&= \Big (\frac{1+c_2}{1-c_2} \Big ) \| \g^{(k)} \|^2_{M^{-1}} + \beta^2_k \|\p^{(k-1)}\|_{M}^2. \label{eq:6}
\end{align}
Let $c_3 = (1+c_2) / (1-c_2) \geq 1$. 
The recursive inequality \eqref{eq:6} with the initial vector $\p^{(0)} = -M^{-1} {\g^{(0)}}$ implies
\begin{equation} \label{eq:p_k_a}
\| \p^{(k)} \|^2_{M} \leq c_3 \bigg( \|\g^{(k)}\|^2_{M^{-1}} + \sum_{i=0}^{k-1} \Big(\prod_{j=i+1}^k \beta_j^2\Big) \|{\g^{(i)}}\|_{M^{-1}}^2 \bigg).
\end{equation}
It can be derived from the expression of $\beta_k$ in \eqref{eq:beta} that
\begin{equation} \label{eq:7}
\prod_{j=i+1}^k \beta_j^2 = \frac{\| {\g^{(k)}} \|^4_{M^{-1}}}{\| {\g^{(i)}} \|^4_{M^{-1}}}.  
\end{equation}
By substituting the product of $\beta_k$ in \eqref{eq:p_k_a} into \eqref{eq:7}, we obtain
\begin{equation} \label{eq:4.31}
\| \p^{(k)} \|^2_{M} \leq c_3 \| \g^{(k)} \|^4_{M^{-1}} \sum_{i=0}^{k} \frac{1}{\| \g^{(i)} \|^2_{M^{-1}}}.
\end{equation}
Since $\nabla E_\rI$ is Lipschitz continuous by Lemma~\ref{lma:Lipschitz}, there exists $\overline{\gamma} > 0$ such that $\|\g^{(k)}\|_{M^{-1}} \leq \bar{\gamma}$.
Moreover, by \eqref{eq:1}, we have
\begin{equation}
\sum_{i=0}^k \frac{1}{\| \g^{(i)} \|^2_{M^{-1}}}  \leq  \frac{k+1 }{\varepsilon^2}.
\label{eq:4.32}
\end{equation}
Consequently, from \eqref{eq:4.31}, we have
\begin{equation*}
\| \p^{(k)} \|^2_{M} \leq \frac{c_3 \overline{\gamma}^4}{\varepsilon^2} (k+1).
\end{equation*}
This implies
\[
\sum_{k=0}^{\infty} \frac{1}{\| \p^{(k)} \|^2_{M}} \geq \frac{\varepsilon^2}{c_3 \overline{\gamma}^4} \sum_{k=0}^{\infty} \frac{1}{k+1} = \infty
\]
which contradicts \eqref{eq:4}.
\end{proof}

\section{Numerical experiments}
\label{sec:6}

In this section, we present the numerical results of the proposed preconditioned nonlinear CG method for IEM, as described in Algorithm~\ref{alg:IEM}. The benchmark tetrahedral mesh models shown in Figure~\ref{fig:Models} are sourced from Jacobson's GitHub~\cite{Jacobson}, Gu's website~\cite{GuOMT}, and the BraTS databases~\cite{BaGB21,BaAS17}. 
Some tetrahedral mesh models are generated using the \texttt{iso2mesh}~\cite{FaBo09,TrYF20} and \texttt{JIGSAW} mesh generators~\cite{Engw14,Engw15,Engw16,EnIv14,EnIv16}. All experiments are implemented in MATLAB on a laptop with an AMD Ryzen 9 5900HS processor and 32 GB of RAM.

\begin{figure}[h]
\centering
\resizebox{\textwidth}{!}{
\begin{tabular}{ccccc}
\hline
Model name & Arnold & Heart & Igea & Brain\\
$\#\K^0(\M)$ & 6,990 & 18,408 & 22,930 & 27,834 \\
$\#\K^3(\M)$ & 36,875 & 103,751 & 130,375 & 158,583\\
&
\includegraphics[height=3.5cm]{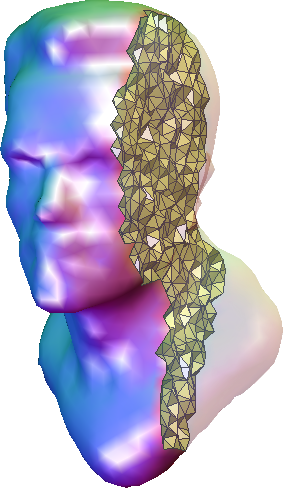} &
\includegraphics[height=3.5cm]{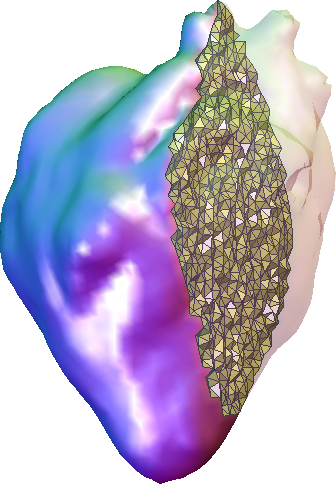} &
\includegraphics[height=3.5cm]{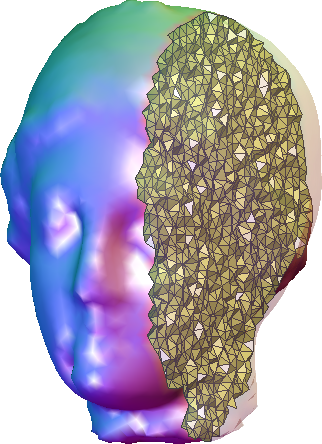} &
\includegraphics[height=3.5cm]{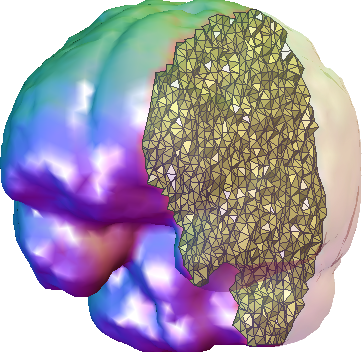} \\
\hline\\
\hline
Model name & Lion & David Head & Max Planck & Apple\\
$\#\K^0(\M)$ & 36,727 & 40,669 & 66,935 & 102,906 \\
$\#\K^3(\M)$ & 207,774 & 233,663 & 390,361 & 559,122\\
&\includegraphics[height=3.5cm]{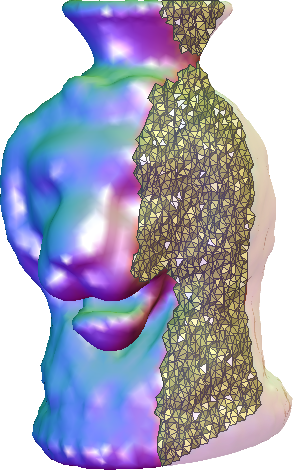} &
\includegraphics[height=3.5cm]{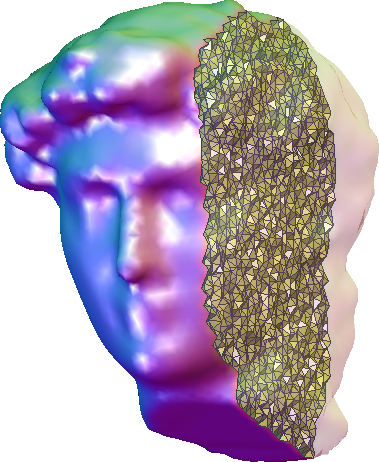} &
\includegraphics[height=3.5cm]{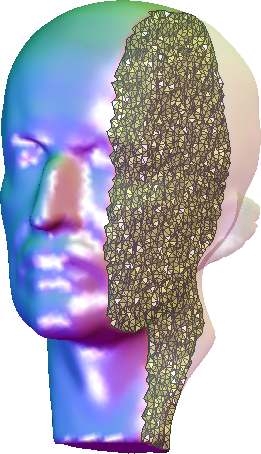} &
\includegraphics[height=3.5cm]{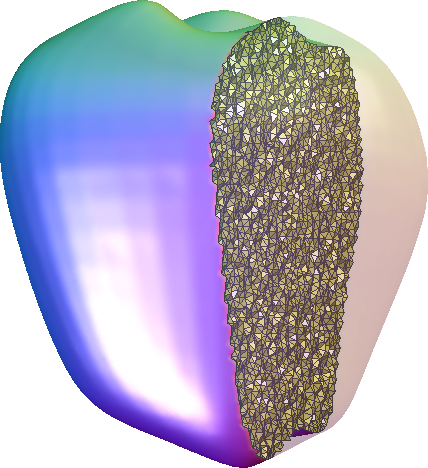} \\ 
\hline
\end{tabular}
}
\caption{The benchmark tetrahedral mesh models.}
\label{fig:Models}
\end{figure}

\subsection{Initialization and step length}
\label{sec:6.1}
We first examine the effect of the initial solution and the step-length selection. At the beginning of Algorithm~\ref{alg:IEM}, we carry out the fixed-point iteration of volumetric stretch energy minimization (VSEM) \cite[Algorithm~4.4]{YuLL19} for 15 iterations to obtain an initial solution. This VSEM method reduces the objective functional $E_\rI$ rapidly in the early iterations, but its effectiveness diminishes after roughly ten iterations, as demonstrated by Figure~\ref{fig:initial_step} (left), where we ran the VSEM method for 100 iterations (blue) on the Arnold model. As a result, we use the VSEM method for 15 iterations to provide the initial map and then switch to the preconditioned nonlinear CG method, which continues to decrease $E_\rI$ (orange).

For the step length of the line-search method, as described in Section~\ref{sec:4}, we first approximate the one-dimensional function $\varphi$ by the quadratic model $\varphi_q$ in \eqref{eq:quad_model}. If the resulting step does not satisfy the sufficient-decrease condition \eqref{eq:decreas_cond}, we then refine it using the cubic model.
Figure~\ref{fig:initial_step} (right) shows the first iteration of Algorithm~\ref{alg:IEM} on the Arnold model, where both the quadratic model $\varphi_q$ (orange) and the cubic model $\varphi_c$ (yellow) closely match the line-search curve $\varphi$ (blue). Thus, the minimizers of $\varphi_q$ and $\varphi_c$ both provide good estimates of the true minimizer, and the cubic model is rarely needed. In practice, we take $0.9 \times \mathrm{argmin}~\varphi_q$, which satisfies \eqref{eq:decreas_cond} in almost all of our numerical experiments.

\begin{figure}[h]
\centering
\resizebox{\textwidth}{!}{
\begin{tabular}{cc}
\includegraphics[height=4.5cm]{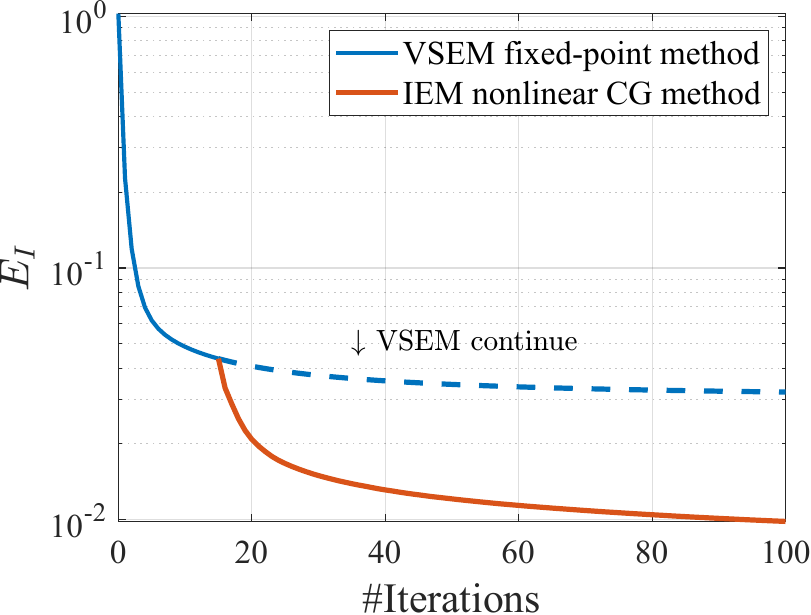} &
\includegraphics[height=4.5cm]{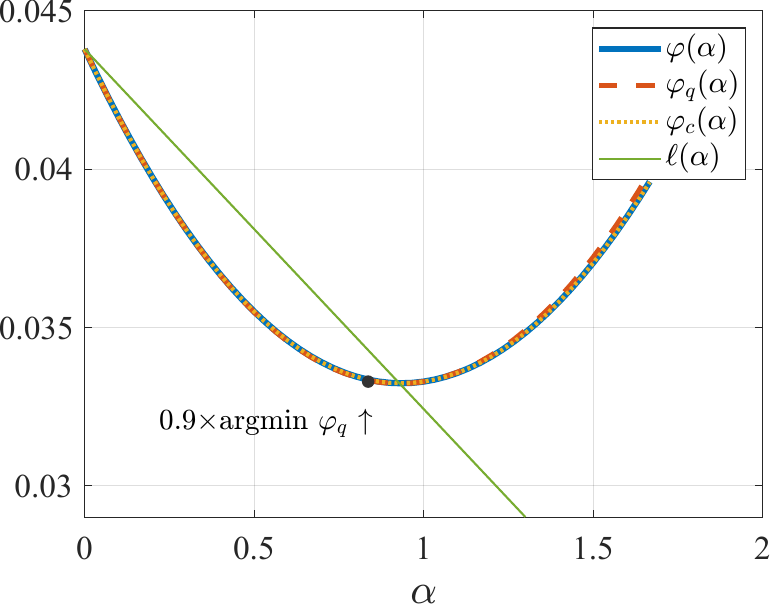} 
\end{tabular}
}
\caption{
Left: Isovolumetric energy $E_\rI$ in \eqref{eq:Ea1} over iterations for VSEM \cite[Algorithm~4.4]{YuLL19} (blue) and for Algorithm~\ref{alg:IEM} (orange) initialized by VSEM. Right: Line-search curve at the first iteration of Algorithm~\ref{alg:IEM} on the Arnold model, together with its quadratic and cubic approximations \eqref{eq:quad_model}, \eqref{eq:cubic_model}, and the sufficient decrease condition \eqref{eq:decreas_cond}.
}
\label{fig:initial_step}
\end{figure}

\subsection{Corrective parameter and preconditioning}
\label{sec:6.2}
We next examine the effect of different choices of the corrective parameter $\beta_k$ in \eqref{eq:beta} for our IEM problem \eqref{eq:IEM}.
We test Algorithm~\ref{alg:IEM} with four standard nonlinear CG formulas for $\beta_k$ \cite{FlRe64,PoRi69,DaYu99,HaZh05}, both with and without preconditioning.
In particular, we consider the CG direction without preconditioning
\[
\p^{(k+1)} = -\g^{(k+1)} + \beta_k \p^{(k)}, \qquad \p^{(0)} = -\g^{(0)},
\]
with the following $\beta_k$, given by
\begin{align}
\beta_k^{\mathrm{FR}} &= \frac{\|\g^{(k+1)}\|^2}{\|\g^{(k)}\|^2}, \qquad
\beta_k^{\mathrm{PR}} = \frac{\g^{(k+1)^\top} \y_k}{\|\g^{(k)}\|^2}, \qquad
\beta_k^{\mathrm{DY}} = \frac{\|\g^{(k+1)}\|^2}{\p^{(k)^\top} \y^{(k)}}, \nonumber\\
&\beta_k^{\mathrm{N}} = \bigg( \y^{(k)} - 2 \p^{(k)} \frac{\|\y^{(k)}\|^2}{\p^{(k)^\top} \y^{(k)}}\bigg)^\top \frac{\g^{(k+1)}}{\p^{(k)^\top} \y^{(k)}},
\label{eq:beta_CG}
\end{align}
where $\y^{(k)} = \g^{(k+1)} - \g^{(k)}$.
On the other hand, the CG direction with preconditioning
\[
\p^{(k+1)} = - M^{-1}\g^{(k+1)} + \bar{\beta}_k \p^{(k)}, \qquad \p^{(0)} = -M^{-1}\g^{(0)},
\]
is tested with the following $\bar{\beta}_k$, given by
\begin{align}
\bar{\beta}_k^{\mathrm{FR}} &= \frac{\|\g^{(k+1)}\|^2_{M^{-1}}}{\|\g^{(k)}\|^2_{M^{-1}}}, \qquad
\bar{\beta}_k^{\mathrm{PR}} = \frac{\g^{(k+1)^\top} M^{-1} \y_k}{\|\g^{(k)}\|^2_{M^{-1}}}, \qquad
\bar{\beta}_k^{\mathrm{DY}} = \frac{\|\g^{(k+1)}\|^2_{M^{-1}}}{\p^{(k)^\top} \y^{(k)}}, \nonumber\\
&\bar{\beta}_k^{\mathrm{N}} = \bigg( M^{-1} \y^{(k)} - 2 \p^{(k)} \frac{\|\y^{(k)}\|^2}{\p^{(k)^\top} \y^{(k)}}\bigg)^\top \frac{\g^{(k+1)}}{\p^{(k)^\top} \y^{(k)}}.
\label{eq:beta_PCG}
\end{align}

The numerical results of isovolumetric energy $E_\rI$ obtained using different corrective parameters \eqref{eq:beta_CG} and \eqref{eq:beta_PCG} across all benchmark models are summarized in Table~\ref{tab:beta_pred}. Unexpectedly, all four variants of parameters produced broadly similar results, both with and without preconditioning. Yet, the preconditioned variants noticeably improved the outcomes relative to those without preconditioning. It is worth noting that this constant preconditioning increased the computational time by only about $7\%$, reflecting the efficiency provided by the preordered Cholesky decomposition.

While a constant preconditioner is efficient and compatible with our global convergence result, the nonlinearity of $E_\rI$ suggests that updating the preconditioner at each iteration may better approximate the Hessian and improve practical convergence.
We therefore compare Algorithm~\ref{alg:IEM} with the fixed preconditioner $M$ in \eqref{eq:precond} against an updated variant
\begin{equation}
M_k = 
\begin{bmatrix}
I_3 \otimes [L_\rV(f^{(k)})]_{\I,\I} \\
& I_2 \otimes [L_\rV(f^{(k)})]_{\B,\B}
\end{bmatrix}.
\label{eq:precond_vary}
\end{equation}
Figure~\ref{fig:precondition} (left) illustrates the convergence of Algorithm~\ref{alg:IEM} using the preconditioners $M$ and $M_k$ on the Lion model over 200 iterations. The updated $M_k$ attains a lower value of $E_\rI$ at the early stage, but its performance tends to level off as the iterations progress. Therefore, a practical trade-off is to update the preconditioner only during the initial iterations and then keep it fixed for the remainder of the process. Figure~\ref{fig:precondition} (right) compares three strategies: updating the preconditioner at each iteration, updating it for the first 25 iterations, and updating it for the first 50 iterations. As expected, extending the initial update phase leads to a lower value of $E_\rI$ and yields results similar to updating it at each iteration.

\begin{table}[tbp]
\centering
\caption{
Isovolumetric energy $E_\rI$ after 100 iterations of Algorithm~\ref{alg:IEM}, comparing runs with and without preconditioning and with different CG parameter $\beta_k$ \eqref{eq:beta_CG} and \eqref{eq:beta_PCG}.
}
\label{tab:beta_pred}
\resizebox{\textwidth}{!}{
\begin{tabular}{lcccc c cccc}
\hline
\multirow{2}{*}{Model name}& \multicolumn{4}{c}{non-preconditioning} && \multicolumn{4}{c}{preconditioning} \\
\cline{2-5}  \cline{7-10} 
& $\beta_k^{\mathrm{FR}}$ & $\beta_k^{\mathrm{PR}}$ & $\beta_k^{\mathrm{DY}}$ & $\beta_k^{\mathrm{N}}$ & &  $\bar{\beta}_k^{\mathrm{FR}}$ & $\bar{\beta}_k^{\mathrm{PR}}$ & $\bar{\beta}_k^{\mathrm{DY}}$ & $\bar{\beta}_k^{\mathrm{N}}$  \\
\hline
Arnold        & $1.5\times 10^{-2}$ & $1.5\times 10^{-2}$ & $1.5\times 10^{-2}$ & $1.5\times 10^{-2}$ && $9.6\times 10^{-3}$ & $9.2\times 10^{-3}$ & $9.3\times 10^{-3}$ & $9.4\times 10^{-3}$ \\ 
Heart         & $3.8\times 10^{-3}$ & $3.8\times 10^{-3}$ & $3.8\times 10^{-3}$ & $3.9\times 10^{-3}$ && $3.2\times 10^{-3}$ & $3.1\times 10^{-3}$ & $3.2\times 10^{-3}$ & $3.1\times 10^{-3}$ \\ 
Igea          & $2.7\times 10^{-3}$ & $2.7\times 10^{-3}$ & $2.7\times 10^{-3}$ & $2.7\times 10^{-3}$ && $2.4\times 10^{-3}$ & $2.4\times 10^{-3}$ & $2.4\times 10^{-3}$ & $2.4\times 10^{-3}$ \\ 
Brain         & $4.2\times 10^{-3}$ & $4.2\times 10^{-3}$ & $4.2\times 10^{-3}$ & $4.3\times 10^{-3}$ && $3.9\times 10^{-3}$ & $3.8\times 10^{-3}$ & $3.8\times 10^{-3}$ & $3.8\times 10^{-3}$ \\ 
Lion          & $1.3\times 10^{-2}$ & $1.3\times 10^{-2}$ & $1.3\times 10^{-2}$ & $1.4\times 10^{-2}$ && $8.2\times 10^{-3}$ & $8.0\times 10^{-3}$ & $8.1\times 10^{-3}$ & $8.0\times 10^{-3}$ \\ 
David Head    & $4.8\times 10^{-3}$ & $4.8\times 10^{-3}$ & $4.8\times 10^{-3}$ & $4.9\times 10^{-3}$ && $4.2\times 10^{-3}$ & $4.2\times 10^{-3}$ & $4.2\times 10^{-3}$ & $4.2\times 10^{-3}$ \\ 
Max Planck    & $2.9\times 10^{-3}$ & $2.9\times 10^{-3}$ & $2.9\times 10^{-3}$ & $3.0\times 10^{-3}$ && $2.2\times 10^{-3}$ & $2.1\times 10^{-3}$ & $2.2\times 10^{-3}$ & $2.1\times 10^{-3}$ \\ 
Apple         & $4.6\times 10^{-4}$ & $4.6\times 10^{-4}$ & $4.6\times 10^{-4}$ & $4.6\times 10^{-4}$ && $3.8\times 10^{-4}$ & $3.8\times 10^{-4}$ & $3.8\times 10^{-4}$ & $3.8\times 10^{-4}$ \\ 
\hline
\end{tabular}
}
\end{table}

\begin{figure}[h]
\centering
\resizebox{\textwidth}{!}{
\begin{tabular}{cc}
\includegraphics[height=4.5cm]{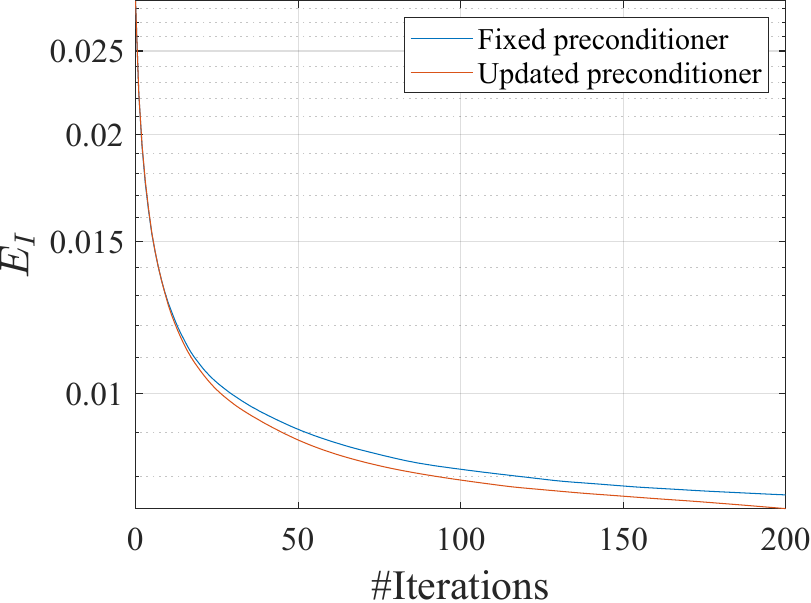} &
\includegraphics[height=4.5cm]{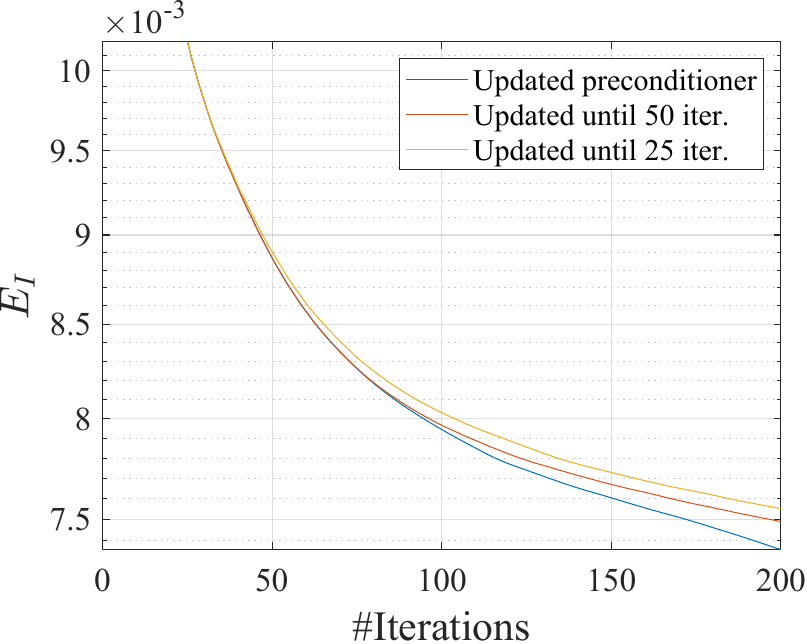} 
\end{tabular}
}
\caption{
Left: Isovolumetric energy $E_\rI$ over 200 iterations on the Lion model for Algorithm~\ref{alg:IEM} with fixed preconditioner $M$ \eqref{eq:precond} and updated preconditioner $M_k$ in \eqref{eq:precond_vary}.
Right: Effect of the preconditioning update strategy on the Lion model: updated every iteration, updated only for the first 25 or 50 iterations.
}
\label{fig:precondition}
\end{figure}

\subsection{Comparison to the state-of-the-art method}
We now evaluate the volume preservation of the mappings produced by Algorithm~\ref{alg:IEM} and compare the results with those produced by a state-of-the-art method.

For each tetrahedron $\tau\in\K^3(\M)$, we quantify the volume distortion by
\begin{equation} \label{eq:VolDist}
D_\rV(f,\tau) = \left|\frac{|f(\tau)| / \mathcal{V}(f) - |\tau| / \mathcal{V}(e)}{|\tau| / \mathcal{V}(e)} \right|.
\end{equation}
Figure~\ref{fig:HistVolDist} shows the histograms of \(D_{\rV}(f,\tau)\) for the mappings obtained by Algorithm~\ref{alg:IEM} after 500 iterations. Most tetrahedra exhibit distortions below 0.1, indicating good local volume preservation.

We next compare our method against the VSEM by the fixed-point iteration method \cite[Algorithm~4.4]{YuLL19}, which has been reported to outperform an OMT-based approach \cite[Algorithm 4]{SuCL17}. Table~\ref{tab:Comparison} summarizes results for Algorithm~\ref{alg:IEM} and VSEM over all benchmark models after 500 iterations, including the mean and standard deviation (SD) of $D_\rV(f,\tau)$, the isovolumetric energy $E_\rI$ \eqref{eq:Ea}, computational time cost, and the number of folded tetrahedra. The same data are visualized in Figures~\ref{fig:Energy_Time} and \ref{fig:Vol_dist}.

Across all models, Algorithm~\ref{alg:IEM} consistently achieves a lower value of $E_\rI$, along with reduced mean and standard deviation of $D_{\rV}(f,\tau)$, which indicates better volume preservation compared to the VSEM method. Moreover, our approach is also less time-consuming than the VSEM method, especially for large meshes. This efficiency arises because the VSEM method requires solving a linear system at each iteration, whereas our approach utilizes a constant preconditioner with preordered Cholesky factors applied consistently throughout the process.  Furthermore, in our experiments, Algorithm~\ref{alg:IEM} produced no folded tetrahedra, providing numerical evidence of bijectivity, while VSEM occasionally exhibited foldovers.

In summary, Algorithm~\ref{alg:IEM} demonstrates improved accuracy and efficiency compared to the VSEM method~\cite[Algorithm~4.4]{YuLL19}.

\begin{figure}[tbp]
\centering
\resizebox{\textwidth}{!}{
\begin{tabular}{cccc}
Arnold & Heart & Igea & Brain \\ 
\begin{overpic}[height=4cm]{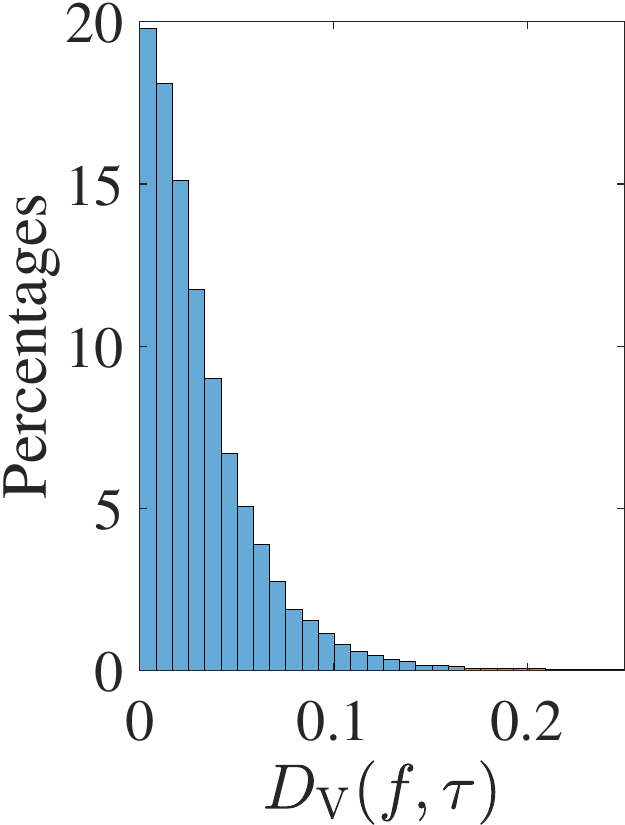}
\put(50,52){\includegraphics[height=1.5cm]{Arnold_v6990.png}}
\end{overpic} &
\begin{overpic}[height=4cm]{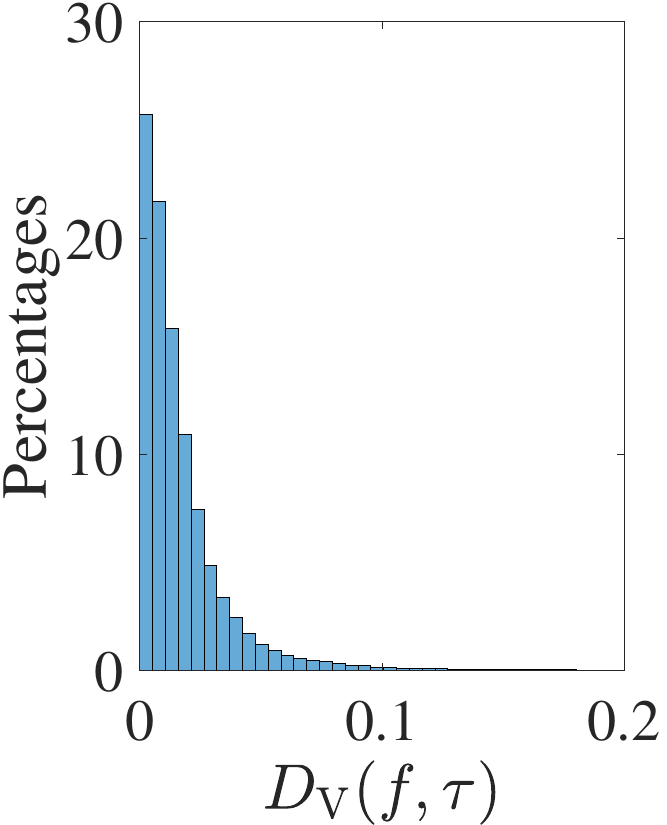}
\put(45,53){\includegraphics[height=1.5cm]{Heart_v18408.png}}
\end{overpic} &
\begin{overpic}[height=4cm]{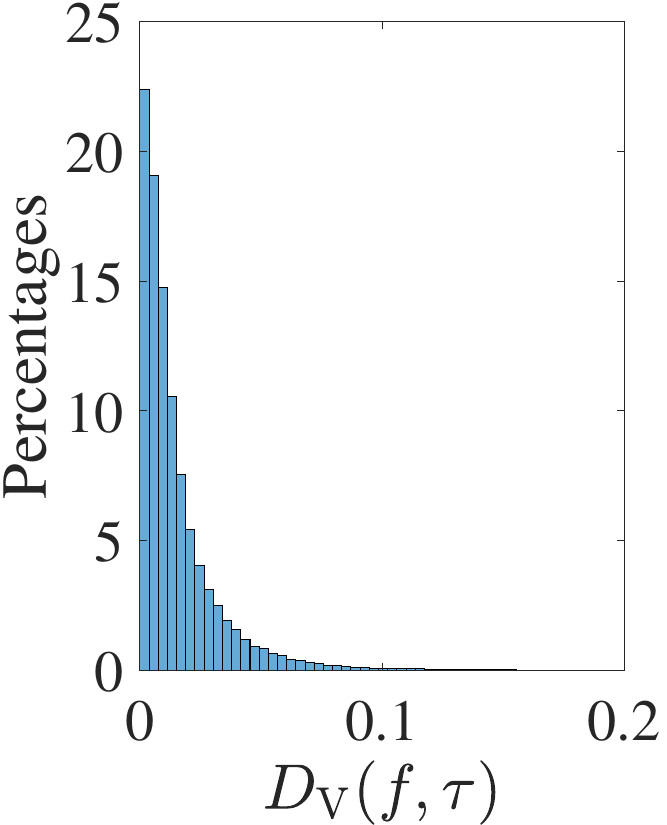}
\put(44,53){\includegraphics[height=1.5cm]{Igea_v22930.png}}
\end{overpic} &
\begin{overpic}[height=4cm]{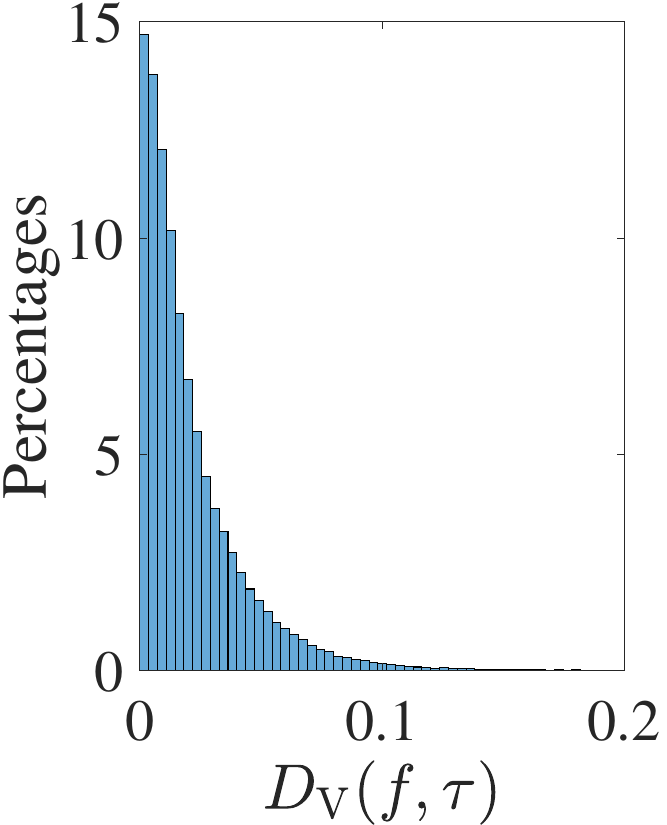} 
\put(40,60){\includegraphics[height=1.2cm]{Brain_v27834.png}}
\end{overpic} \\
Lion & David Head & Max Planck & Apple\\
\begin{overpic}[height=4cm]{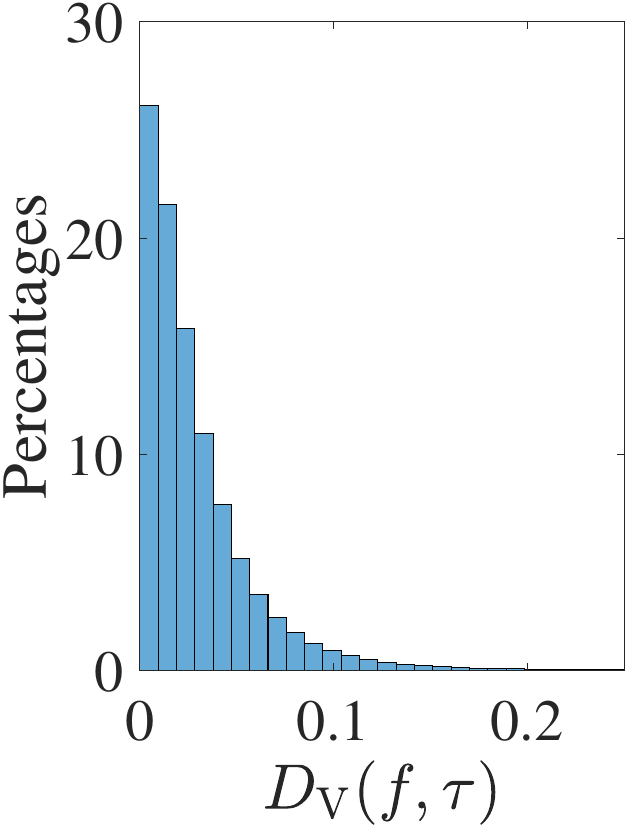}
\put(48,52){\includegraphics[height=1.5cm]{Lion_v36727.png}}
\end{overpic} &
\begin{overpic}[height=4cm]{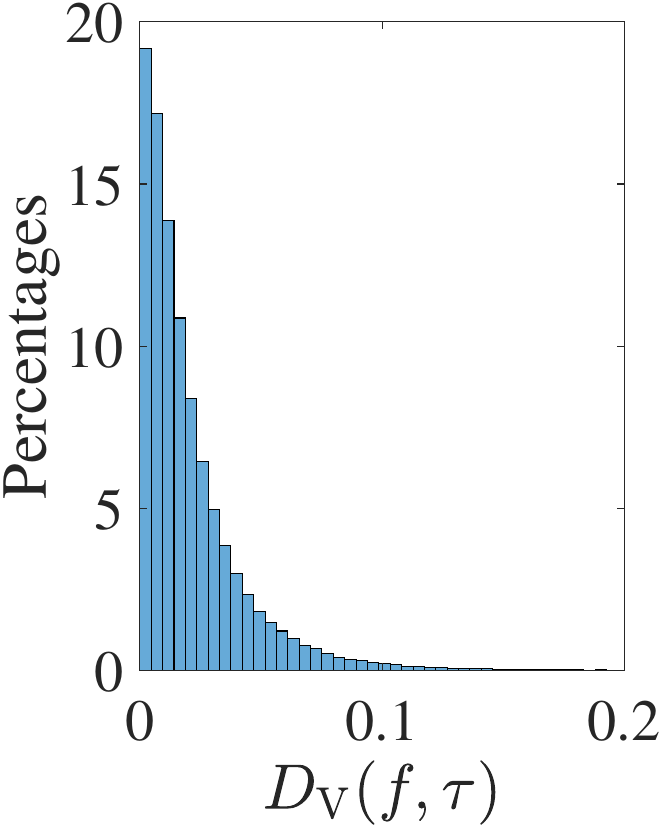}
\put(43,55){\includegraphics[height=1.4cm]{DavidHead_v40669.png}}
\end{overpic} &
\begin{overpic}[height=4cm]{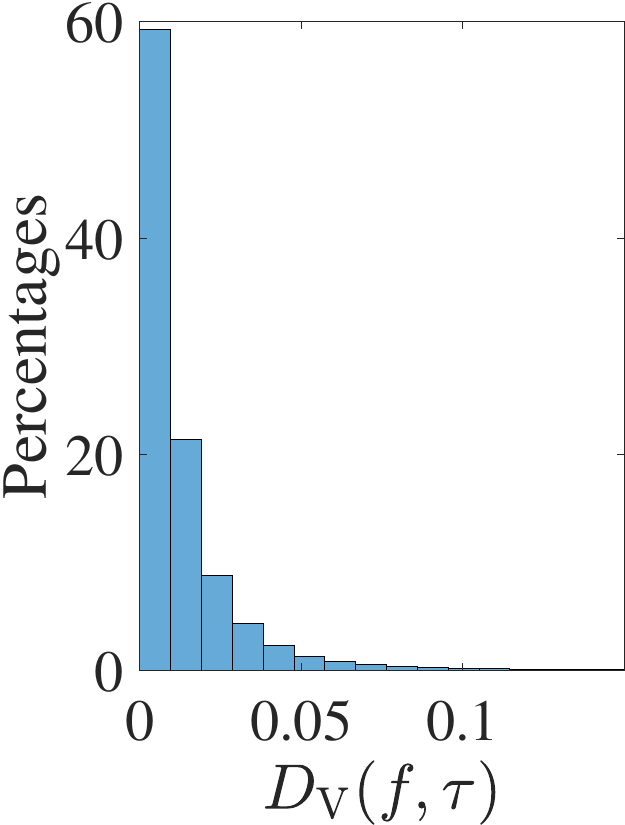}
\put(50,52){\includegraphics[height=1.5cm]{MaxPlanck_v66935.png}}
\end{overpic} &
\begin{overpic}[height=4cm]{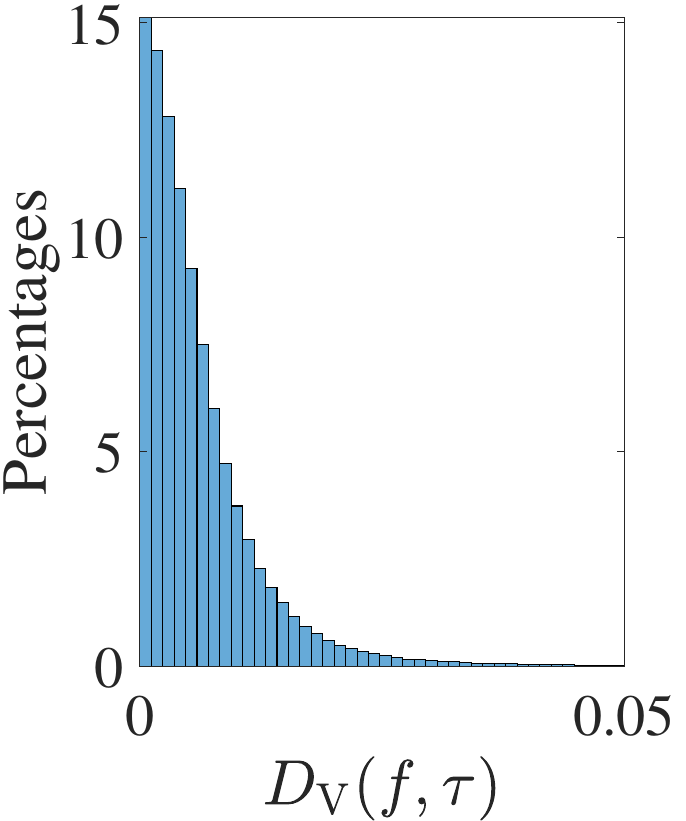}
\put(44,60){\includegraphics[height=1.2cm]{Apple_v102906.png}}
\end{overpic}
\end{tabular}
}
\caption{Histograms of the local volume distortion $D_\mathrm{V}(f,\tau)$ \eqref{eq:VolDist} for mappings computed by Algorithm~\ref{alg:IEM}.}
\label{fig:HistVolDist}
\end{figure}

\begin{table}[tbp]
\centering
\caption{
Numerical comparison of the proposed IEM by Algorithm~\ref{alg:IEM} and the VSEM by the fixed-point iteration~\cite[Algorithm~4.4]{YuLL19} after 500 iterations.
$D_\rV(f,\tau)$: per-tetrahedron volume distortion \eqref{eq:VolDist}. $E_\rI$: isovolumetric energy \eqref{eq:Ea}.
}
\label{tab:Comparison}
\resizebox{\textwidth}{!}{
\begin{tabular}{lcccrc c cccrc}
\hline
\multirow{3}{*}{Model name} & \multicolumn{5}{c}{IEM by Algorithm \ref{alg:IEM} }  &&\multicolumn{5}{c}{VSEM by fixed-point iteration\cite[Alg. 4.4]{YuLL19}} \\
\cline{2-6}  \cline{8-12} 
& \multicolumn{2}{c}{$D_\rV(f,\tau)$} & \multirow{2}{*}{$E_{\rI}(f)$} & Time & \#Fold- && \multicolumn{2}{c}{$D_\rV(f,\tau)$} & \multirow{2}{*}{$E_{\rI}(f)$} & Time & \#Fold- \\ 
 & Mean & SD & & (sec.) & ings && Mean & SD & & (sec.) & ings
\\
\hline
Arnold           & $3.2\times 10^{-2}$ & $3.1\times 10^{-2}$ & $7.7\times 10^{-3}$ & 20 & 0 && $6.2\times 10^{-2}$ & $6.9\times 10^{-2}$ & $3.0\times 10^{-2}$ & 20 & 2 \\ 
Heart            & $1.7\times 10^{-2}$ & $2.1\times 10^{-2}$ & $2.8\times 10^{-3}$ & 80 & 0 && $2.3\times 10^{-2}$ & $3.4\times 10^{-2}$ & $5.8\times 10^{-3}$ & 109 & 1 \\ 
Igea             & $1.5\times 10^{-2}$ & $1.8\times 10^{-2}$ & $2.2\times 10^{-3}$ & 100 & 0 && $1.9\times 10^{-2}$ & $2.7\times 10^{-2}$ & $4.1\times 10^{-3}$ & 162 & 0 \\ 
Brain            & $2.1\times 10^{-2}$ & $2.2\times 10^{-2}$ & $3.5\times 10^{-3}$ & 136 & 0 && $2.6\times 10^{-2}$ & $2.9\times 10^{-2}$ & $5.5\times 10^{-3}$ & 212 & 0 \\ 
Lion             & $2.9\times 10^{-2}$ & $3.1\times 10^{-2}$ & $6.9\times 10^{-3}$ & 181 & 0 && $4.2\times 10^{-2}$ & $5.3\times 10^{-2}$ & $1.8\times 10^{-2}$ & 252 & 3 \\ 
David Head       & $2.1\times 10^{-2}$ & $2.2\times 10^{-2}$ & $3.6\times 10^{-3}$ & 209 & 0 && $2.6\times 10^{-2}$ & $3.5\times 10^{-2}$ & $7.3\times 10^{-3}$ & 350 & 0 \\ 
Max Planck       & $1.3\times 10^{-2}$ & $1.8\times 10^{-2}$ & $1.9\times 10^{-3}$ & 371 & 0 && $1.7\times 10^{-2}$ & $2.9\times 10^{-2}$ & $4.1\times 10^{-3}$ & 947 & 6 \\ 
Apple            & $6.0\times 10^{-3}$ & $6.2\times 10^{-3}$ & $3.1\times 10^{-4}$ & 538 & 0 && $8.1\times 10^{-3}$ & $1.0\times 10^{-2}$ & $6.6\times 10^{-4}$ & 1746 & 0 \\ 
\hline
\end{tabular}
}
\end{table}

\begin{figure}[tbp]
\centering
\resizebox{\textwidth}{!}{
\begin{tabular}{cc}
\includegraphics[height=4cm]{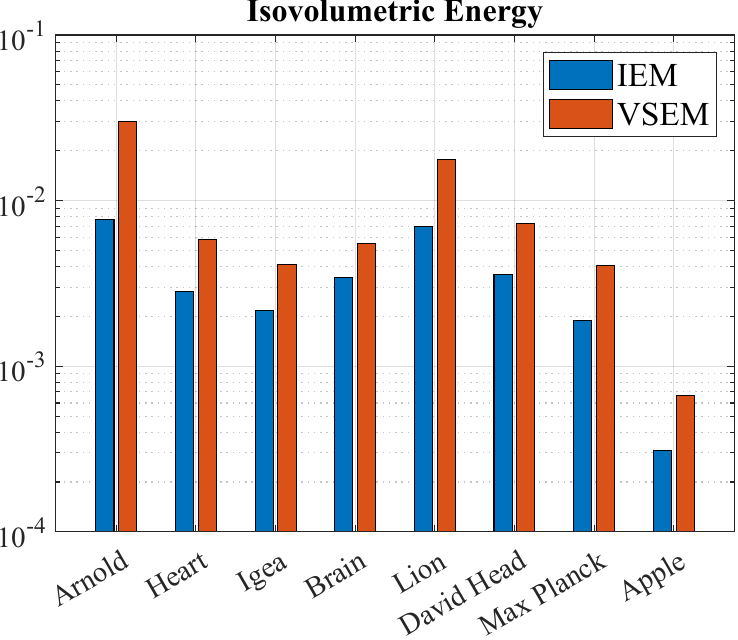} &
\includegraphics[height=4cm]{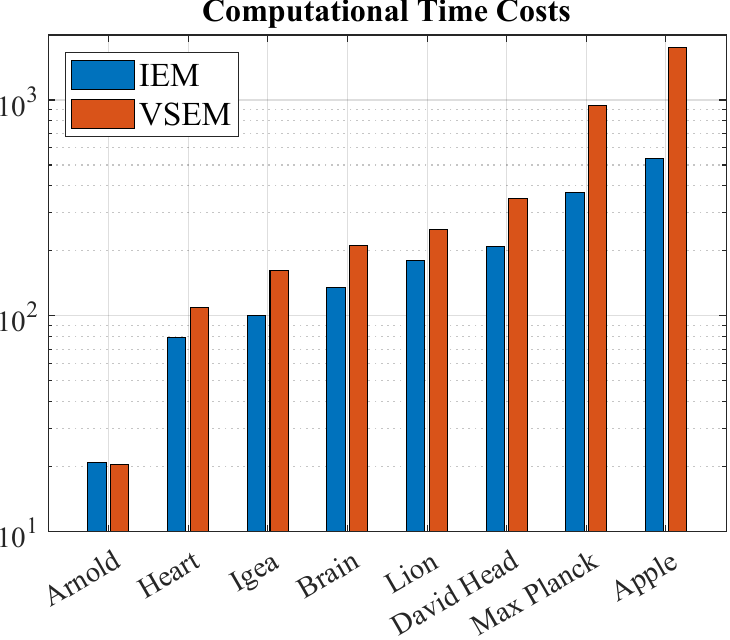} 
\end{tabular}
}
\caption{
Comparison of the proposed IEM by Algorithm~\ref{alg:IEM} with the VSEM by the fixed-point iteration \cite[Algorithm~4.4]{YuLL19} on all benchmark models. Left: isovolumetric energy $E_\rI$ in \eqref{eq:Ea}. Right: computational time.
}
\label{fig:Energy_Time}
\end{figure}

\begin{figure}[tbp]
\centering
\resizebox{\textwidth}{!}{
\begin{tabular}{cc}
\includegraphics[height=4cm]{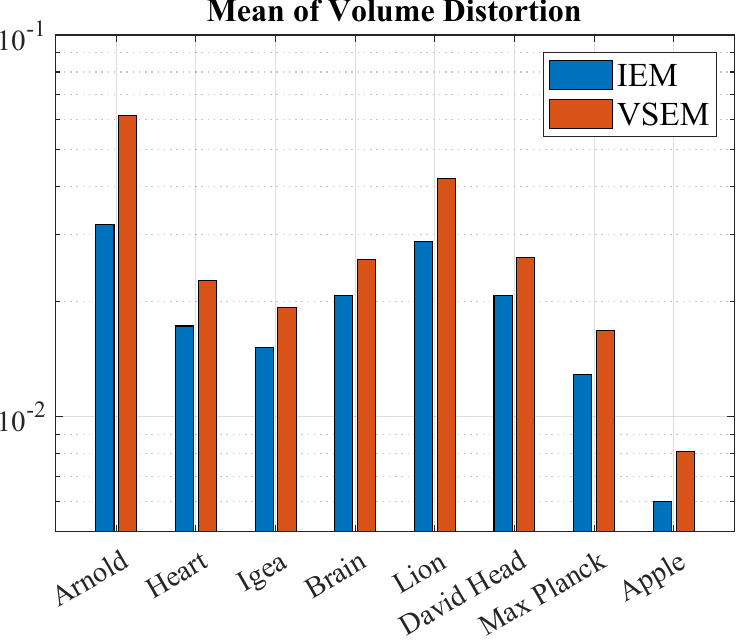} &
\includegraphics[height=4cm]{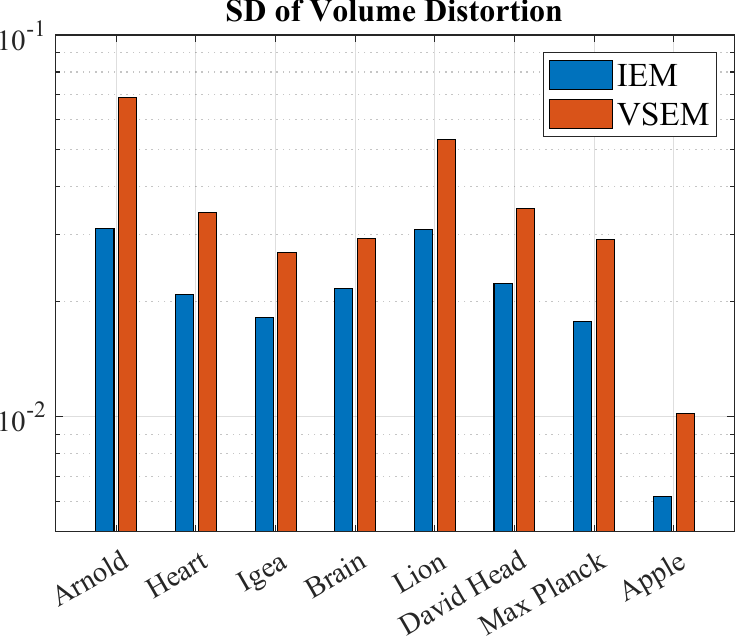} 
\end{tabular}
}
\caption{
Comparison of the proposed IEM by Algorithm~\ref{alg:IEM} with the VSEM by the fixed-point iteration \cite[Algorithm~4.4]{YuLL19} on all benchmark models. Left: mean volume distortion $D_\rV(f,\tau)$ in \eqref{eq:VolDist}. Right: standard deviation of the volume distortion $D_\rV(f,\tau)$.
}
\label{fig:Vol_dist}
\end{figure}

\section{Applications to shape registration and deformation}
\label{sec:7}

The isovolumetric energy is invariant under rotations. Consequently, the computed parameterizations are defined up to a global rotation in $SO(3)$. A one-to-one correspondence between two parameterized manifolds can therefore be established by determining an optimal rotation in $SO(3)$.


Specifically, for two tetrahedral meshes $\M_0$ and $\M_1$, we first apply Algorithm \ref{alg:IEM} to compute volume-preserving parameterizations $f_s: \M_s\to\R^3$, $s=0,1$. 
We then remove the remaining $\mathrm{SO}(3)$ freedom by solving
$$
g_s = \argmin_{g \in \mathrm{SO}(3)} \sum_{v\in\K^0(\mathcal{M}_s)} \| g \circ f_s(v) - v  \|_2^2 \,\mathrm{vol}(v), ~ s=0,1, 
$$
where $\mathrm{vol}(v)$ denotes the vertex volume of $v$ defined by $\mathrm{vol}(v) = \frac{1}{4} \sum_{\tau\ni v} |\tau|$. A rigid alignment on the unit ball is given by $h = g_1^{-1} \circ g_0$. Ultimately, the volume-preserving registration $\varphi:\mathcal{M}_0\to\mathcal{M}_1$ between the two manifolds is given by $\varphi = f_1^{-1}\circ h\circ f_0$, as illustrated in Figure~\ref{fig:BrainReg}.

\begin{figure}[tbp]
\center
\resizebox{\textwidth}{!}{
\begin{tabular}{ccc}
$\mathcal{M}_0$ \includegraphics[height=2.5cm]{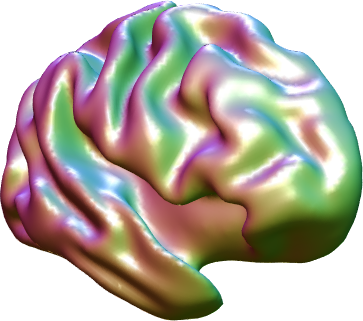} &
\raisebox{1.2cm}{$\xrightarrow{~\varphi \,=\, f_1^{-1}\circ\, h\,\circ f_0~}$} & 
\includegraphics[height=2.5cm]{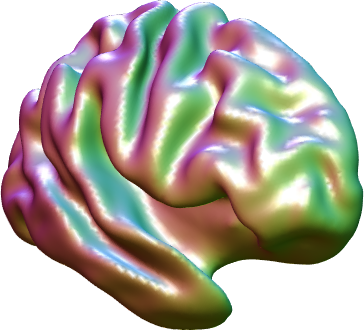} $\mathcal{M}_1$ \\
$f_0\Bigg\downarrow$ & & $f_1\Bigg\downarrow$ \\[0.5cm]
\includegraphics[height=2.5cm]{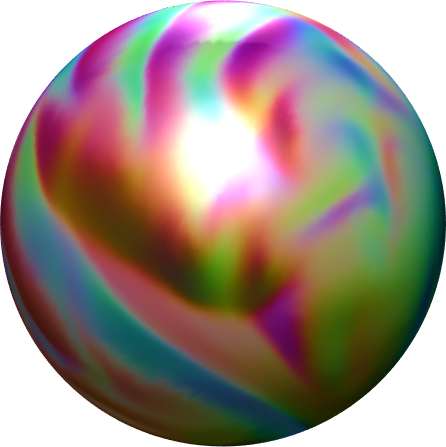} &
\raisebox{1.2cm}{$\xrightarrow{~~~h \,=\, g_1^{-1} \circ \,g_0~~~}$} & 
\includegraphics[height=2.5cm]{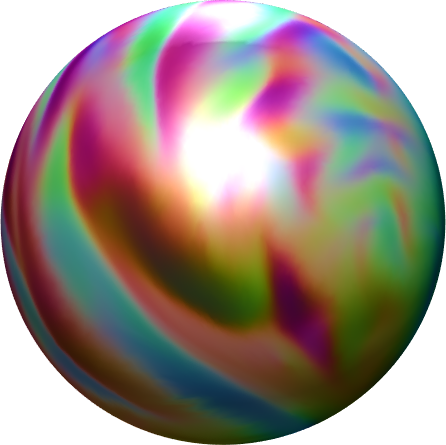}
\end{tabular}
}
\caption{An illustration of the computational procedures of a volume-preserving registration mapping between two simplicial $3$-manifolds.}
\label{fig:BrainReg}
\end{figure}

To visualize the deformation induced by the mapping $\varphi$, one can employ the linear homotopy, defined as
\begin{equation} \label{eq:homotopy}
\mathcal{H}(v,t) = (1-t) v + t \varphi(v).
\end{equation}
In Figure~\ref{fig:BrainHomotopy}, we demonstrate the simplicial $3$-manifolds $\mathcal{H}(\mathcal{M}_0,t)$ for $t=0, \frac{1}{3}, \frac{2}{3}$, and $1$, which show that the deformation is very subtle because the shapes of the two given simplicial $3$-manifolds are similar.

\begin{figure}[tbp]
\centering
\resizebox{\textwidth}{!}{
\begin{tabular}{cccc}
\includegraphics[height=2.5cm]{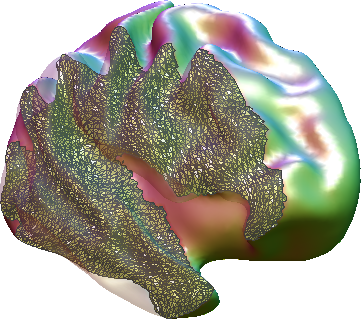} &
\includegraphics[height=2.5cm]{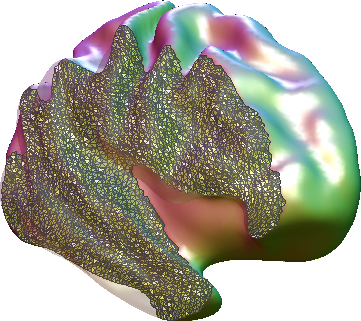} &
\includegraphics[height=2.5cm]{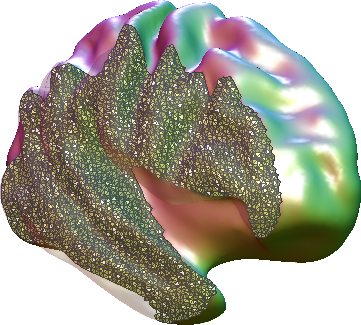} &
\includegraphics[height=2.5cm]{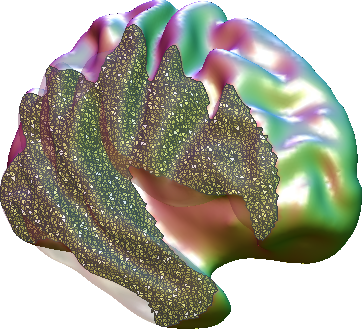} \\
$\mathcal{H}(\mathcal{M}_0,0)$ & 
$\mathcal{H}(\mathcal{M}_0,\frac{1}{3})$ &
$\mathcal{H}(\mathcal{M}_0,\frac{2}{3})$ &
$\mathcal{H}(\mathcal{M}_0,1)$
\end{tabular}
}
\caption{The image of linear homotopy $\mathcal{H}(\mathcal{M}_0,t)$ between the identity mapping and the volume-preserving registration mapping, as defined in \eqref{eq:homotopy}.}
\label{fig:BrainHomotopy}
\end{figure}

\section{Concluding remarks} 
\label{sec:8}
In this paper, we introduced an isovolumetric energy functional for measuring the volume distortion of simplicial mappings. By expressing boundary vertices in spherical coordinates, we reformulated the computation of ball-shaped volume-preserving parameterizations as an unconstrained nonlinear optimization problem. This formulation enables the use of a preconditioned nonlinear CG method, whose global convergence under a strong Wolfe line search is also established. Numerical experiments on all benchmark models demonstrate that the proposed method achieves better volume preservation and higher efficiency than the previous state-of-the-art VSEM method. Finally, we illustrate the practical utility of the method through volume-preserving registration and deformation between manifolds.

\section*{Data Availability}
The MATLAB demo is available at \url{http://tiny.cc/SphereIEM}. Additional data are available from the corresponding authors upon request.


\section*{Acknowledgements}
This work was supported by the National Science and Technology Council, the National Center for Theoretical Sciences, and the Higher Education Sprout Project--Center for Optimal Intelligent Data Analytics and Prediction.

\section*{Conflict of interest} 
The authors declare that there are no conflicts of interest.


\end{document}